\pdfoutput=1
\documentclass{article}

\usepackage[colorlinks,linkcolor=black,citecolor=black]{hyperref}
\usepackage{bm}
\usepackage{bbm}
\usepackage{mathrsfs}
\usepackage{amsmath}
\usepackage{amsthm}
\usepackage{amssymb}
\usepackage{mathdots}
\usepackage{multirow}
\usepackage{geometry}
\usepackage{mathtools}
\usepackage{fancyhdr}
\usepackage{cite}
\usepackage{authblk}
\usepackage[pagewise]{lineno}

\allowdisplaybreaks
\numberwithin{equation}{section}

\newcommand{\keywords}[1]{\textit{\quad Keywords:}#1}

\newtheorem{theorem}{Theorem}[section]
\newtheorem{lemma}[theorem]{Lemma}

\newtheorem{definition}[theorem]{Definition}

\newtheorem{remark}[theorem]{Remark}

\newcommand{\dd}{\mathrm{d}}
\newcommand{\re}{\mathrm{Re}}

\title{Stability of Poiseuille Flow of Navier-Stokes Equations on \(\mathbb{R}^2\)}
\author{Zhile Li \thanks{School of Mathematical Sciences, University of Science and Technology of China, Hefei, 230026, People's Republic of China. lizhile@mail.ustc.edu.cn}}
\date{}

\begin{document}
	
	\maketitle
	
	\begin{abstract}
		We consider solutions to the Navier–Stokes equations on \(\mathbb{R}^2\) close to the Poiseuille flow with viscosity \(0< \nu < 1\). For the linearized problem, we prove that when the \(x\)-frequency satisfies \(|k| \ge \nu^{-\frac{1}{3}}\), the perturbation decays on a time-scale proportional to \(\nu^{-\frac{1}{2}}|k|^{-\frac{1}{2}}\). Since it decays faster than the heat equation, this phenomenon is referred to as enhanced dissipation. Then we concern the non-linear equations. We show that if the initial perturbation \(\omega_{in}\) is at most of size \(\nu^\frac{7}{3}\) in an anisotropic Sobolev space, then the size of the perturbation remains no more than twice the size of its initial value.
	\end{abstract}

	\keywords{Navier–Stokes equations, Poiseuille flow, stability}
	
	\section{Introduction}
	
	Consider the incompressible Navier-Stokes equations on \(\mathbb{R}^2\)
	\begin{equation}
		\label{eq 1.1}
		\begin{cases*}
			\partial_t U + (U \cdot \nabla) U + \nabla P - \nu \Delta U = 0,\\
			\nabla \cdot U = 0,
		\end{cases*}
	\end{equation}
	where \(U = (U_1, U_2)\) is the velocity vector field, \(P\) is the scalar pressure, and the kinematic viscosity \(\nu > 0\) is inversely proportional to the Reynolds number. Defining \(\nabla^\bot = (-\partial_y, \partial_x)\) as the rotation of the gradient vector, then the vorticity \(\Omega \coloneqq \nabla^\bot U\) satisfies the active scalar equations
	\begin{equation}
		\label{eq 1.2}
		\begin{cases*}
			\partial_t \Omega + (U \cdot \nabla) \Omega - \nu \Delta \Omega = 0,\\
			U = \nabla^\bot \Psi, \Delta \Psi = \Omega,
		\end{cases*}
	\end{equation}
	where \(\Psi\) is the corresponding stream-function.
	
	The stability of \eqref{eq 1.1} is a prominent research topic in both mathematics and physics. As stated in \cite{ref5}, stability means that, given two Banach spaces \(X\) and \(Y\), we say that a solution 
	\(U_E\) of \eqref{eq 1.1} is quantitatively stable with exponent \(\gamma\) if \(\| U_{in} - U_E \|_X \ll \nu^\gamma\) implies \(\| U(t) - U_E \|_Y \ll 1\) for all \(t > 0\) and \(\| U(t) - U_E \|_Y \to 0\) as \(t \to \infty\).
	
	Enhanced dissipation is another important topic related to \eqref{eq 1.1}. Enhanced dissipation refers to the phenomenon in certain dynamical systems where the scalar field exhibits dissipation at a rate much faster than that of the heat equation. This behavior is associated with features like shear flow or rotational flow. Enhanced dissipation has been extensively studied in the physics literature \cite{ref9,ref14} and has received significant attention from the mathematics community. Quantitative questions have been studied in the context of passive scalars \cite{ref2,ref3,ref4}, the Navier-Stokes equations near the Couette flow \cite{ref7,ref18}, and Lamb–Oseen vortices \cite{ref16}. 
	For 2D linear passive scalars, when \(x\)-frequencies satisfy \(|k| \gg \nu\), the decay rate follows the form 
	\(\exp(-\nu^\frac{1}{3} |k|^\frac{2}{3} t)\) \cite{ref12}. In the setting of linearized periodic Couette flow, Kelvin identified enhanced dissipation with a dissipation time-scale of \(\nu^{-\frac{1}{3}}\) \cite{ref13}.
	
	There has been extensive research on the stability and enhanced dissipation near the so-called Poiseuille flow \cite{ref11,ref14}. The Poiseuille flow is given by
	\[U^P(x, y) = (y^2, 0), P^P(x, y) = 2 \nu x.\]
	This flow is important and one reason is that it is the simplest non-trivial example of a shear flow on \(\mathbb{R}^2\) besides the Couette flow \(U^C = (y,0)\).
	
	Our work also investigates the stability and enhanced dissipation near the Poiseuille flow. By writing \(U = (y^2, 0) + u\), with corresponding \(\Omega = -2y + \omega\), we can rewrite \eqref{eq 1.2} as
	\begin{equation}
		\label{eq 1.3}
		\begin{cases*}
			\partial_t \omega = - y^2 \partial_x \omega + 2 \partial_x \psi + \nu \Delta \omega - u \cdot \nabla \omega,\\
			u = \nabla^\bot \psi, \Delta \psi = \omega.
		\end{cases*}
	\end{equation}
	Here \(u, \omega, \psi\) are thought of as perturbations
	of the velocity, vorticity and stream-function near the Poiseuille flow. Studying the stability of the Poiseuille flow reduces to analyzing the solutions to \eqref{eq 1.3}.
	
	For the linear problem, we prove that when the \(x\)-frequency satisfies \(|k| \ge \nu^{-\frac{1}{3}}\), the perturbation decays on a time-scale proportional to \(\nu^{-\frac{1}{2}}|k|^{-\frac{1}{2}}\). In other words, enhanced dissipation exists. We will see this in detail in Section \ref{sec 2}. For the non-linear problem, we prove a quantitative stability threshold for the initial perturbation \(\omega_{in}\). Namely, we show that for suitable anisotropic Sobolev norms \(\| \cdot \|_X\) and \(\| \cdot \|_Y\), that if \(\| \omega_{in} \|_X \lesssim \nu^\frac{7}{3}\), then the corresponding solution \(\omega\) of \eqref{eq 1.3} satisfies \(\| \omega \|_Y \ll 1\) for all \(t > 0\) and \(\| \omega \|_Y \to 0\) as \(t \to \infty\). Before presenting the main results, we first introduce the following notations.
	
	Given a function \(f(t,x,y)\), we denote its Fourier transform only in \(x\) by
	\[f_k(t,y) = \frac{1}{2\pi}\int_\mathbb{R} f(t,x,y) e^{-ikx} \dd x.\]
	We denote the corresponding Hamiltonian and Laplacian by \(\nabla_k \coloneqq (ik, \partial_y)\) and \(\Delta_k \coloneqq \partial_y^2 - |k|^2\). We use the standard notation \(\langle \cdot \rangle \coloneqq \sqrt{1 + \cdot ^2}\). Given two quantities \(A, B \ge 0\), we write \(A \lesssim B\) to indicate that there exists a constant \(C > 0\) such that \(A \le C B\). If \(A \lesssim B\) and \(B \lesssim A\), we write \(A \approx B\).
	
	Our main result is that
	\begin{theorem}
		\label{thm 1.1}
		Suppose \(\omega_{in}\) is initial datum for \eqref{eq 1.3}. Then for all \(J \in [1, +\infty), m \in (\frac{3}{4}, 1)\), there exists a constant \(\delta > 0\) independent of \(\nu\) such that if
		\begin{align*}
			\epsilon \coloneqq& \| \langle \partial_x \rangle^m \omega_{in} \|_{L_{x,y}^2} + \| \nu^\frac{1}{3} \langle \partial_x \rangle^m \langle \nu^\frac{1}{3} \partial_x \rangle^{-\frac{1}{4}} \nabla \omega_{in} \|_{L_{x,y}^2} + \| \nu^{-\frac{1}{3}} \langle \partial_x \rangle^m \langle \nu^\frac{1}{3} \partial_x \rangle^\frac{1}{4} y \omega_{in} \|_{L_{x,y}^2}\\
			&+ \| \nu^{-\frac{1}{3}} \langle \partial_x \rangle^m \langle \nu^\frac{1}{3} \partial_x \rangle^\frac{1}{4} \partial_y \Delta^{-1} \omega_{in} \|_{L_{x,y}^2} + \| y \omega_{in,k} \|_{L_k^\infty L_y^2} + \| \nabla_k \Delta_k^{-1} \omega_{in,k} \|_{L_k^\infty L_y^2} \le \delta \nu^\frac{7}{3},
		\end{align*}
		then for all \(c > 0\) sufficiently small (independent of \(\nu\) and \(\delta\)) and all \(\nu \in (0, 1)\), the corresponding solution \(\omega\) to \eqref{eq 1.3} satisfies
		\begin{align*}
			&\| \langle c \lambda_\nu^{pl}(\partial_x) t \rangle^J \langle \partial_x \rangle^m \omega \|_{L_{x,y}^2} + \| \nu^\frac{1}{3} \langle c \lambda_\nu^{pl}(\partial_x) t \rangle^J \langle \partial_x \rangle^m \langle \nu^\frac{1}{3} \partial_x \rangle^{-\frac{1}{4}} \nabla \omega \|_{L_{x,y}^2} + \| \nu^{-\frac{1}{3}} \langle c \lambda_\nu^{pl}(\partial_x) t \rangle^J \langle \partial_x \rangle^m \langle \nu^\frac{1}{3} \partial_x \rangle^\frac{1}{4} y \omega \|_{L_{x,y}^2}\\
			&+ \| \nu^{-\frac{1}{3}} \langle c \lambda_\nu^{pl}(\partial_x) t \rangle^J \langle \partial_x \rangle^m \langle \nu^\frac{1}{3} \partial_x \rangle^\frac{1}{4} \partial_y \Delta^{-1} \omega \|_{L_{x,y}^2} + \| y \omega_k \|_{L_k^\infty L_y^2} + \| \nabla_k \Delta_k^{-1} \omega_k \|_{L_k^\infty L_y^2} \le 2 \epsilon.
		\end{align*}
		Here \(\lambda_\nu^{pl}(\partial_x)\) is a Fourier multiplier defined on the Fourier side as
		\[\lambda_\nu^{pl}(k) =
		\begin{cases*}
			 \nu^\frac{1}{2} |k|^\frac{1}{2}, |k| \ge \nu^{-\frac{1}{3}},\\
			 \nu |k|^2, |k| < \nu^{-\frac{1}{3}}.
		\end{cases*}
		\]
		
	\end{theorem}
	
	Most previous results about stability have domains such as \(\mathbb{T} \times \mathbb{R}\) \cite{ref8,ref11,ref15,ref17} and \(\mathbb{T} \times [-1,1]\) \cite{ref6,ref10}. Whereas \cite{ref1} provided the stability results where the \(x\) variable is unbounded and non-periodic. The idea of our paper is just from \cite{ref1}. We shall prove Theorem \ref{thm 1.1} using an energy method similar to that used in proving Theorem 1.1 in \cite{ref1}. Unlike \cite{ref1}, our expression for \(\partial_t \omega\) includes a term \(\partial_x \psi\) related to the stream-function \(\psi\). Consequently, our energy must also include terms involving \(\psi\), making our energy formulation more complex. Fortunately, our \(\partial_t \omega\) contains the term \(y^2 \partial_x \omega\) instead of \(y \partial_x \omega\), so we do not need to introduce the operator \(\mathcal{J}_k\) as in \cite{ref1}, which simplifies our nonlinear estimates.
	
	The remainder of our paper is organized as follows. In Section \ref{sec 2}, we consider the linear equation and prove its stability. In Section \ref{sec 3}, we estimate the nonlinear term and ultimately prove Theorem \ref{thm 1.1}.
	
	\section{Linearized Problem}
	\label{sec 2}
	
	Taking the Fourier transform in \(x\) in \eqref{eq 1.3}, we obtain that
	\begin{equation}
		\label{eq 2.1}
		\begin{cases*}
			\partial_t \omega_k = - iky^2 \omega_k + 2ik \psi_k + \nu \Delta_k \omega_k - (\nabla^\bot \psi \cdot \nabla \omega)_k,\\
			\Delta_k \psi_k = \omega_k.
		\end{cases*}
	\end{equation}
	By removing the non-linear terms from \eqref{eq 2.1}, we get the following linear equations
	\begin{equation}
		\label{eq 2.2}
		\begin{cases*}
			\partial_t \omega_k = - iky^2 \omega_k + 2ik \psi_k + \nu \Delta_k \omega_k,\\
			\Delta_k \psi_k = \omega_k.
		\end{cases*}
	\end{equation}
	
	The main result for the linear problem is that
	\begin{theorem}
		\label{thm 2.1}
		We define the energy \(E_k\) that depends on \(k\)
		\[E_k = \frac{1}{2} \| \omega_k \|_2^2 + \frac{1}{2} c_\alpha \alpha_k \| \nabla_k \omega_k \|_2^2 + 2 c_\beta \beta_k \re \langle iky \omega_k, \partial_y \omega_k \rangle + \frac{1}{2} c_\gamma \gamma_k (\| y \omega_k \|_2^2 + 2 \| \nabla_k \psi_k \|_2^2),\]
		and the corresponding dissipation \(D_k\)
		\[D_k = c_\gamma \gamma_k \nu \| \omega_k \|_2^2 + \nu \| \nabla_k \omega_k \|_2^2 + c_\alpha \alpha_k \nu \| \Delta_k \omega_k \|_2^2 + 4 c_\beta \beta_k |k|^2 \| y \omega_k \|_2^2 + c_\gamma \gamma_k \nu \| y \nabla_k \omega_k \|_2^2 + 8 c_\beta \beta_k |k|^2 \| \partial_y \psi_k \|_2^2,\]
		where
		\[\alpha_k =
		\begin{cases*}
			\nu^\frac{1}{2}|k|^{-\frac{1}{2}}, |k| \ge \nu^{-\frac{1}{3}},\\
			\nu^\frac{2}{3}, |k| < \nu^{-\frac{1}{3}},
		\end{cases*}
		\beta_k =
		\begin{cases*}
			|k|^{-1}, |k| \ge \nu^{-\frac{1}{3}},\\
			\nu^\frac{1}{3}, |k| < \nu^{-\frac{1}{3}},
		\end{cases*}
		\gamma_k =
		\begin{cases*}
			\nu^{-\frac{1}{2}}|k|^\frac{1}{2}, |k| \ge \nu^{-\frac{1}{3}},\\
			\nu^{-\frac{2}{3}}, |k| < \nu^{-\frac{1}{3}},
		\end{cases*}
		\]
		and constants \(c_\alpha, c_\beta, c_\gamma\) satisfy
		\[c_\beta - c_\alpha^2 > 0, c_\gamma - \frac{8 c_\beta^2}{c_\alpha} > 0.\]
		Then
		\begin{equation}
			\label{eq 2.3}
			E_k \approx \| \omega_k \|_2^2 + \alpha_k \| \nabla_k \omega_k \|_2^2 + \gamma_k \| y \omega_k \|_2^2 + \gamma_k \| \partial_y \psi_k \|_2^2.
		\end{equation}
		If \(\omega_k, \psi_k\) solve \eqref{eq 2.2}, then there exists \(c > 0\) sufficiently small such that
		\begin{equation}
			\label{eq 2.4}
			\frac{\dd E_k}{\dd t} \le -4cD_k -4c \lambda_k E_k,
		\end{equation}
		where
		\[
		\lambda_k =
		\begin{cases*}
			\nu^\frac{1}{2} |k|^\frac{1}{2}, |k| \ge \nu^{-\frac{1}{3}},\\
			\nu |k|^2, |k| < \nu^{-\frac{1}{3}}.
		\end{cases*}
		\]
	\end{theorem}
	
	\begin{remark}
		Applying the Gronwall inequality to \eqref{eq 2.4}, we obtain
		\[E_k(t) \le e^{-4c \lambda_k t} E_k(0).\]
		When \(|k| \ge \nu^{-\frac{1}{3}}\), the dissipation rate of \(E_k\) is faster than that of the heat equation's \(e^{-\nu t}\), meaning that enhanced dissipation exists.
	\end{remark}
	
	Taking the time derivative of each term in \(E_k\), we have the following lemma
	
	\begin{lemma}
		\label{lem 2.3}
		If \(\omega_k, \psi_k\) solve \eqref{eq 2.2}, then the following equalities hold.
		\begin{equation}
			\label{eq 2.5}
			\frac{\dd}{\dd t} \| \omega_k \|_2^2 = -2\nu \| \nabla_k \omega_k \|_2^2,
		\end{equation}
		\begin{equation}
			\label{eq 2.6}
			\frac{\dd}{\dd t} \| \nabla_k \omega_k \|_2^2 = - 2\nu \| \Delta_k \omega_k \|_2^2 - 4 \re \langle iky \omega_k, \partial_y \omega_k \rangle,
		\end{equation}
		\begin{equation}
			\label{eq 2.7}
			\frac{\dd}{\dd t} \re \langle iky \omega_k, \partial_y \omega_k \rangle = -2|k|^2 \| y \omega_k \|_2^2 - 4|k|^2 \| \partial_y \psi_k \|_2^2 - 2\nu \re \langle \Delta_k \omega_k, iky \partial_y \omega_k \rangle,
		\end{equation}
		\begin{equation}
			\label{eq 2.8}
			\frac{\dd}{\dd t} \| y \omega_k \|_2^2 = 2\nu \| \omega_k \|_2^2 -2\nu \| y \nabla_k \omega_k \|_2^2 - 8 \re \langle iky \psi_k, \partial_y \psi_k \rangle,
		\end{equation}
		\begin{equation}
			\label{eq 2.9}
			\frac{\dd}{\dd t} \| \nabla_k \psi_k \|_2^2 = -2\nu \| \omega_k \|_2^2 + 4 \re \langle iky \psi_k, \partial_y \psi_k \rangle.
		\end{equation}
	\end{lemma}
	
	\begin{remark}
		In particular, combining \eqref{eq 2.8} and \eqref{eq 2.9}, we obtain
		\[\frac{\dd}{\dd t} \| y \omega_k \|_2^2 + 2 \frac{\dd}{\dd t} \| \nabla_k \psi_k \|_2^2 = -2\nu \| \omega_k \|_2^2 - 2 \nu \| y \nabla_k \omega_k \|_2^2 \le 0.\]
		This motivates the structure of the \(\gamma\) term in our definition of \(E_k\).
	\end{remark}
	
	\begin{proof}
		For \eqref{eq 2.5}, using integration by parts, we obtain
		\[\frac{\dd}{\dd t} \| \omega_k \|_2^2 =  2 \re \int_{\mathbb{R}} (iky^2 \omega_k + 2ik \psi_k + \nu \Delta_k \omega_k) \overline{\omega_k} \dd y = -2\nu \| \nabla_k \omega_k \|_2^2.\]
		Noting that \(\| \nabla_k \omega_k \|_2^2 = \| k \omega_k \|_2^2 + \| \partial_y \omega_k \|_2^2\), together with \eqref{eq 2.5}, we have
		\begin{align*}
			\frac{\dd}{\dd t} \| \nabla_k \omega_k \|_2^2 &= \frac{\dd}{\dd t} \| k \omega_k \|_2^2 + \frac{\dd}{\dd t} \| \partial_y \omega_k \|_2^2\\
			&= |k|^2 (-2 \nu \| \nabla_k \omega_k \|_2^2) + 2 \re \int_{\mathbb{R}} \partial_y (iky^2 \omega_k + 2ik \psi_k + \nu \Delta_k \omega_k) \partial_y \overline{\omega_k} \dd y\\
			&= -2\nu \| k \nabla_k \omega_k \|_2^2 + 4 \re \langle iky \omega_k, \partial_y \omega_k \rangle - 2\nu \| \partial_y \nabla_k \omega_k \|_2^2.
		\end{align*}
		Using \(\| \Delta_k \omega_k \|_2^2 = \| \nabla_k^2 \omega_k \|_2^2 = \| k \nabla_k \omega_k \|_2^2 + \| \partial_y \nabla_k \omega_k \|_2^2\), and then \eqref{eq 2.6} follows. For \eqref{eq 2.7}, we directly compute by \eqref{eq 2.2},
		\begin{align*}
			\frac{\dd}{\dd t} \re \langle iky \omega_k, \partial_y \omega_k \rangle =& \re \langle iky (- iky^2 \omega_k + 2ik \psi_k + \nu \Delta_k \omega_k), \partial_y \omega_k \rangle\\
			&+ \re \langle iky \omega_k, \partial_y (- iky^2 \omega_k + 2ik \psi_k + \nu \Delta_k \omega_k) \rangle\\
			=& \nu \re \int_{\mathbb{R}} iky \Delta_k \omega_k \partial_y \overline{\omega_k} + iky \omega_k \partial_y \Delta_k \overline{\omega_k} \dd y\\
			&+ \re \int_{\mathbb{R}} k^2y^3 \omega_k \partial_y \overline{\omega_k} - k^2y \omega_k \partial_y(y^2 \overline{\omega_k}) \dd y\\
			&+ \re \int_{\mathbb{R}} -2k^2y \psi_k \partial_y \overline{\omega_k} + 2k^2y \omega_k \partial_y \overline{\psi_k} \dd y.
		\end{align*}
		We treat the first term by integration by parts as
		\[\nu \re \int_{\mathbb{R}} iky \Delta_k \omega_k \partial_y \overline{\omega_k} + iky \omega_k \partial_y \Delta_k \overline{\omega_k} \dd y = -2\nu \re \langle \Delta_k \omega_k, iky \partial_y \omega_k \rangle,\]
		while for the second term we compute
		\[\re \int_{\mathbb{R}} k^2y^3 \omega_k \partial_y \overline{\omega_k} - k^2y \omega_k \partial_y(y^2 \overline{\omega_k}) \dd y = - 2|k|^2 \| y \omega_k \|_2^2.\]
		Lastly, turning to the third term, we have
		\[\re \int_{\mathbb{R}} -2k^2y \psi_k \partial_y \overline{\omega_k} + 2k^2y \omega_k \partial_y \overline{\psi_k} \dd y =  -4|k|^2 \| \partial_y \psi_k \|_2^2,\]
		and \eqref{eq 2.7} follows. For \eqref{eq 2.8}, we have
		\begin{align*}
			\frac{\dd}{\dd t} \| y \omega_k \|_2^2 &= 2 \re \int_{\mathbb{R}} y(- iky^2 \omega_k + 2ik \psi_k + \nu \Delta_k \omega_k) \overline{y \omega_k} \dd y\\
			&= 4 \re \int_{\mathbb{R}} iky^2 \psi_k \Delta_k \overline{\psi_k} \dd y + 2\nu \re \int_{\mathbb{R}} y^2 \Delta_k \omega_k \overline{\omega_k} \dd y\\
			&= -8 \re \langle iky \psi_k, \partial_y \psi_k \rangle -4\nu \re \int_{\mathbb{R}} y \partial_y \omega_k  \overline{\omega_k} \dd y - 2\nu \| y \nabla_k \omega_k \|_2^2.
		\end{align*}
		Notice
		\[\int_{\mathbb{R}} y \partial_y \omega_k  \overline{\omega_k} \dd y = - \int_{\mathbb{R}} (\overline{\omega_k} + y \partial_y \overline{\omega_k}) \omega_k \dd y = - \| \omega_k \|_2^2 - \int_{\mathbb{R}} y \partial_y \overline{\omega_k}\omega_k \dd y,\]
		which implies that
		\[\re \int_{\mathbb{R}} y \partial_y \omega_k  \overline{\omega_k} \dd y = \frac{1}{2} \left( \int_{\mathbb{R}} y \partial_y \omega_k  \overline{\omega_k} \dd y  + \int_{\mathbb{R}} y \partial_y \overline{\omega_k}\omega_k \dd y \right) = -\frac{1}{2} \| \omega_k \|_2^2.\]
		Hence
		\[\frac{\dd}{\dd t} \| y \omega_k \|_2^2 = 2\nu \| \omega_k \|_2^2 -2\nu \| y \nabla_k \omega_k \|_2^2 - 8 \re \langle iky \psi_k, \partial_y \psi_k \rangle.\]
		Turning to \eqref{eq 2.9}, we use \(\Delta_k \psi_k = \omega_k\) to compute
		\[\frac{\dd}{\dd t} \| \nabla_k \psi_k \|_2^2 = - 2\re \int_{\mathbb{R}} \partial_t \Delta_k \psi_k \overline{\psi_k} \dd y = -2 \re \int_{\mathbb{R}} (- iky^2 \omega_k + 2ik \psi_k + \nu \Delta_k \omega_k) \overline{\psi_k} \dd y = -2 \nu \| \omega_k \|_2^2 + 2 \re \int_{\mathbb{R}} iky^2 \Delta_k \psi_k \overline{\psi_k} \dd y,\]
		so \eqref{eq 2.9} follows from
		\[\re \int_{\mathbb{R}} iky^2 \Delta_k \psi_k \overline{\psi_k} \dd y = \re \int_{\mathbb{R}} - \partial_y(iky^2 \overline{\psi_k}) \partial_y \psi_k - ik^3y^2 \psi_k \overline{\psi_k} \dd y = 2 \re \langle iky \psi_k, \partial_y \psi_k \rangle. \qedhere\]
	\end{proof}
	
	Now, we prove Theorem \ref{thm 2.1}.
	
	\begin{proof}
		By Cauchy-Schwarz inequality, we have
		\[|2 c_\beta \beta_k \re \langle iky \omega_k, \partial_y \omega_k \rangle| \le \frac{c_\alpha \alpha_k}{4} \| \nabla_k \omega_k \|_2^2 + \frac{4 c_\beta^2 \beta_k^2}{c_\alpha \alpha_k} |k|^2 \| y \omega_k \|_2^2.\]
		In order to prove \eqref{eq 2.3}, we observe that the missing term \(\|k \psi_k \|_2\) may be bounded from as follows. Since
		\[\re \langle \partial_y \psi_k, y \omega_k \rangle = \re \langle \partial_y \psi_k, y \partial_y^2 \psi_k \rangle - \re \langle \partial_y \psi_k, y k^2 \psi_k \rangle = \frac{1}{2} \| k \psi_k \|_2^2 - \frac{1}{2} \| \partial_y \psi_k \|_2^2,\]
		it follows that
		\[\| k \psi_k \|_2^2 - \| \partial_y \psi_k \|_2^2 = 2 \re \langle \partial_y \psi_k, y \omega_k \rangle \le \| \partial_y \psi_k \|_2^2 + \| y \omega_k \|_2^2,\]
		which implies
		\[\| k \psi_k \|_2^2 \le 2 \| \partial_y \psi_k \|_2^2 + \| y \omega_k \|_2^2.\]
		Thus we obtain the upper and the lower bounds of \(E_k\), and then \eqref{eq 2.3} holds through the definitions of \(\alpha_k, \beta_k, \gamma_k\) and the choices of \(c_\alpha, c_\beta, c_\gamma\).
		
		By the equalities from Lemma \ref{lem 2.3}, \(E_k\) satisfies
		\begin{align*}
			\frac{\dd E_k}{\dd t} =& - \left( c_\gamma \gamma_k \nu \| \omega_k \|_2^2 + \nu \| \nabla_k \omega_k \|_2^2 + c_\alpha \alpha_k \nu \| \Delta_k \omega_k \|_2^2 + 4 c_\beta \beta_k |k|^2 \| y \omega_k \|_2^2 + c_\gamma \gamma_k \nu \| y \nabla_k \omega_k \|_2^2 + 8 c_\beta \beta_k |k|^2 \| \partial_y \psi_k \|_2^2 \right)\\
			&- 2 c_\alpha \alpha_k \re \langle iky \omega_k, \partial_y \omega_k \rangle - 4 c_\beta \beta_k \nu \re \langle \Delta_k \omega_k, iky \partial_y \omega_k \rangle.
		\end{align*}
		The first term is just \(-D_k\). To absorb the last two terms, we note that
		\[- 2 c_\alpha \alpha_k \re \langle iky \omega_k, \partial_y \omega_k \rangle \le \frac{\nu}{4} \| \nabla_k \omega_k \|_2^2 + \frac{4 c_\alpha^2 \alpha_k^2}{\nu} |k|^2 \| y \omega_k \|_2^2,\]
		and
		\[- 4 c_\beta \beta_k \nu \re \langle \Delta_k \omega_k, iky \partial_y \omega_k \rangle \le \frac{c_\alpha \alpha_k \nu}{2} \| \Delta_k \omega_k \|_2^2 + \frac{8 c_\beta^2 \beta_k^2 \nu}{c_\alpha \alpha_k} |k|^2 \| y \nabla_k \omega_k \|_2^2.\]
		Thus
		\begin{align*}
			0 \ge& \frac{\dd E_k}{\dd t} + c_\gamma \gamma_k \nu \| \omega_k \|_2^2 + \frac{3\nu}{4} \| \nabla_k \omega_k \|_2^2 + \frac{c_\alpha \alpha_k \nu}{2} \| \Delta_k \omega_k \|_2^2 + \left( 4 c_\beta \beta_k - \frac{4 c_\alpha^2 \alpha_k^2}{\nu} \right) |k|^2 \| y \omega_k \|_2^2\\
			&+ \left( c_\gamma \gamma_k \nu - \frac{8 c_\beta^2 \beta_k^2 \nu}{c_\alpha \alpha_k} |k|^2 \right) \| y \nabla_k \omega_k \|_2^2 + 8 c_\beta \beta_k |k|^2 \| \partial_y \psi_k \|_2^2\\
			\eqqcolon& \frac{\dd E_k}{\dd t} + I.
		\end{align*}
		It suffices to prove \(\lambda_k E_k \lesssim I\) since \(D_k \lesssim I\). We will show this in both the \(|k| \ge \nu^{-\frac{1}{3}}\) case and the \(|k| \le \nu^{-\frac{1}{3}}\) case.
		
		If \(|k| \ge \nu^{-\frac{1}{3}}\), we compute \(I\) and \(\lambda_k E_k\) explicitly and then we obtain \(\lambda_k E_k \lesssim I\). Now we consider the \(|k| \le \nu^{-\frac{1}{3}}\) case. Taking Fourier transform in \(y\), we can get
		\[\| k \nabla_k \omega_k \|_2^2 \lesssim \| \Delta_k \omega_k \|_2^2, \| k \partial_y \psi_k \|_2^2 \lesssim \| \Delta_k \psi_k \|_2^2 = \| \omega_k \|_2^2.\]
		Using them and \eqref{eq 2.3} to estimate \(\lambda_k E_k\),
		\begin{align*}
			\lambda_k E_k &\lesssim \nu \| k \omega_k \|_2^2 + \alpha_k \nu \| k \nabla_k \omega_k \|_2^2 + \gamma_k \nu |k|^2 \| y \omega_k \|_2^2 + \gamma_k \nu \| k \partial_y \psi_k \|_2^2\\
			&\lesssim \nu \| \nabla_k \omega_k \|_2^2 + \alpha_k \nu \| \Delta_k \omega_k \|_2^2 + \nu^\frac{1}{3} |k|^2 \| y \omega_k \|_2^2 + \gamma_k \nu \| \omega_k \|_2^2\\
			&\lesssim I. \qedhere
		\end{align*}
	\end{proof}

	\section{Non-linear Problem}
	\label{sec 3}
	
	For the non-linear problem, the non-linear term \(-(\nabla^\bot \psi \cdot \nabla \omega)_k\) combines multiple frequencies. Hence the energy \(E_k\) that depends on \(k\) is not sufficient. Thus we define the following energies and dissipations.
	
	\begin{definition}
		We define the energies
		\[\mathcal{E} = \mathcal{E}_1 + \mathcal{E}_2 = \int_{\mathbb{R}} \frac{\langle c \lambda_k t \rangle^{2J}}{M_k(t)} \langle k \rangle^{2m} E_k \dd k + \sup_{k \in \mathbb{R}} \left( \frac{1}{2} \| y \omega_k \|_2^2 + \| \nabla_k \psi_k \|_2^2 \right),\]
		and the dissipations
		\[\mathcal{D} = \mathcal{D}_1 + \mathcal{D}_2 = \int_{0}^{t} \widetilde{\mathcal{D}} \dd s + \sup_{k \in \mathbb{R}} \int_{0}^{t} \nu \| \omega_k \|_2^2 + \nu \| y \nabla_k \omega_k \|_2^2 \dd s,\]
		where
		\[\widetilde{\mathcal{D}} = \int_{\mathbb{R}} \frac{\langle c \lambda_k s \rangle^{2J}}{M_k(s)} \langle k \rangle^{2m} D_k \dd k,\]
		\(J \in [1, +\infty), m \in (\frac{3}{4}, 1)\), \(M_k(t)\) solves the ODE
		\[\left\{
		\begin{aligned}
			&M_k'(t) = cJ^2 \lambda_k \frac{(c \lambda_k t)^2}{\langle c \lambda_k t \rangle^4} M_k(t),\\
			&M_k(0) = 1.
		\end{aligned}
		\right.\]
	\end{definition}
	
	\begin{remark}
		The constant \(c\) in the definition is a sufficiently small positive constant determined by Theorem \ref{thm 2.1}. Importantly, the smallness condition on \(c\) is independent of \(\nu, J, m\).
	\end{remark}
	
	\begin{remark}
		The multiplier \(M_k(t)\) is included in \(\mathcal{E}_1\) to address terms in \(\frac{\dd \mathcal{E}_1}{\dd t}\) which arise from the time-derivative falling on \(\langle c \lambda_k t \rangle^{2J}\). Solving the ODE explicitly, we can get \(1 \le M_k(t) \le e^\frac{\pi J^2}{4}\) for all \(t \ge 0\) and all \(k \in \mathbb{R}\). Thus \(M_k(t)\) is uniformly bounded above and below in both \(k\) and \(t\).
	\end{remark}

	For \(\omega_k\) solving \eqref{eq 2.1}, we define
	\[L_k = - iky^2 \omega_k + 2ik \psi_k + \nu \Delta_k \omega_k,\]
	and
	\[NL_k = - (\nabla^\bot \psi \cdot \nabla \omega)_k = - \int_{\mathbb{R}} \nabla_{k-k'}^\bot \psi_{k-k'} \cdot \nabla_{k'} \omega_{k'} \dd k',\]
	then
	\[\partial_t \omega_k = L_k + NL_k.\]
	By Newton-Leibniz formula,
	\begin{align}
		\label{eq 3.1}
		\mathcal{E}(t) + \mathcal{D}_2(t) =& \mathcal{E}_1(0) + \int_{0}^{t} \frac{\dd \mathcal{E}_1}{\dd s} \dd s + \sup_{k \in \mathbb{R}} \left( \frac{1}{2} \| y \omega_k(0) \|_2^2 + \| \nabla_k \psi_k(0) \|_2^2 + \int_{0}^{t} \frac{\dd}{\dd s} \left( \frac{1}{2} \| y \omega_k \|_2^2 + \| \nabla_k \psi_k \|_2^2 \right) \dd s \right) \nonumber\\
		&+ \sup_{k \in \mathbb{R}} \int_{0}^{t} \nu \| \omega_k \|_2^2 + \nu \| y \nabla_k \omega_k \|_2^2 \dd s \nonumber\\
		\le& \mathcal{E}_1(0) + \int_{0}^{t} \frac{\dd \mathcal{E}_1}{\dd s} \dd s + 2 \sup_{k \in \mathbb{R}} \bigg( \frac{1}{2} \| y \omega_k(0) \|_2^2 + \| \nabla_k \psi_k(0) \|_2^2 + \int_{0}^{t} \frac{\dd}{\dd s} \left( \frac{1}{2} \| y \omega_k \|_2^2 + \| \nabla_k \psi_k \|_2^2 \right) \dd s \nonumber\\
		&+ \int_{0}^{t} \nu \| \omega_k \|_2^2 + \nu \| y \nabla_k \omega_k \|_2^2 \dd s \bigg) \nonumber\\
		\le& 2\mathcal{E}(0) + \int_{0}^{t} \frac{\dd \mathcal{E}_1}{\dd s} \dd s + 2 \sup_{k \in \mathbb{R}} \int_{0}^{t} \frac{\dd}{\dd s} \left( \frac{1}{2} \| y \omega_k \|_2^2 + \| \nabla_k \psi_k \|_2^2 \right) + \nu \| \omega_k \|_2^2 + \nu \| y \nabla_k \omega_k \|_2^2 \dd s.
	\end{align}
	We divide \(\frac{\dd \mathcal{E}_1}{\dd s}\) into the linear part, the nonlinear part, and the terms where the time derivative acts on the multipliers.
	\[\frac{\dd \mathcal{E}_1}{\dd s} = \mathcal{L}_1 + \mathcal{NL}_1 + \int_{\mathbb{R}} Jc \lambda_k \frac{2c \lambda_k s}{\langle c \lambda_k s \rangle^2} \frac{\langle c \lambda_k s \rangle^{2J}}{M_k(s)} \langle k \rangle^{2m} E_k \dd k - \int_{\mathbb{R}} \frac{M_k'(s)}{M_k(s)} \frac{\langle c \lambda_k s \rangle^{2J}}{M_k(s)} \langle k \rangle^{2m} E_k \dd k,\]
	where
	\begin{align*}
		\mathcal{L}_1 =& \int_{\mathbb{R}} \frac{\langle c \lambda_k s \rangle^{2J}}{M_k(s)} \langle k \rangle^{2m} \bigg( \re \langle \omega_k ,L_k \rangle + c_\alpha \alpha_k \re \langle \nabla_k \omega_k, \nabla_k L_k \rangle + 2 c_\beta \beta_k (\re \langle iky L_k, \partial_y \omega_k \rangle + \re \langle iky \omega_k, \partial_y L_k \rangle)\\
		&+ c_\gamma \gamma_k (\re \langle y \omega_k, y L_k \rangle + 2 \re \langle \nabla_k \psi_k , \nabla_k \Delta_k^{-1} L_k \rangle) \bigg) \dd k,
	\end{align*}
	and
	\begin{align*}
		\mathcal{NL}_1 =& \int_{\mathbb{R}} \frac{\langle c \lambda_k s \rangle^{2J}}{M_k(s)} \langle k \rangle^{2m} \bigg( \re \langle \omega_k ,NL_k \rangle + c_\alpha \alpha_k \re \langle \nabla_k \omega_k, \nabla_k NL_k \rangle + 2 c_\beta \beta_k (\re \langle iky NL_k, \partial_y \omega_k \rangle + \re \langle iky \omega_k, \partial_y NL_k \rangle)\\
		&+ c_\gamma \gamma_k ( \re \langle y \omega_k, y NL_k \rangle + 2 \re \langle \nabla_k \psi_k , \nabla_k \Delta_k^{-1} NL_k \rangle) \bigg) \dd k.
	\end{align*}
	By and Theorem \ref{thm 2.1}, we have
	\[\mathcal{L}_1 \le \int_{\mathbb{R}} \frac{\langle c \lambda_k s \rangle^{2J}}{M_k(s)} \langle k \rangle^{2m} (-4cD_k -4c \lambda_k E_k) \dd k = -4c \widetilde{\mathcal{D}} -4 \int_{\mathbb{R}} c \lambda_k \frac{\langle c \lambda_k s \rangle^{2J}}{M_k(s)} \langle k \rangle^{2m} E_k \dd k.\]
	Using Young's inequality and recalling the definition of \(M_k(t)\), we have
	\begin{align*}
		\int_{\mathbb{R}} Jc \lambda_k \frac{2c \lambda_k s}{\langle c \lambda_k s \rangle^2} \frac{\langle c \lambda_k s \rangle^{2J}}{M_k(s)} \langle k \rangle^{2m} E_k \dd k &\le \int_{\mathbb{R}} c \lambda_k \frac{\langle c \lambda_k s \rangle^{2J}}{M_k(s)} \langle k \rangle^{2m} E_k \dd k + \int_{\mathbb{R}} cJ^2 \lambda_k \frac{(c \lambda_k s)^2}{\langle c \lambda_k s \rangle^4} \frac{\langle c \lambda_k s \rangle^{2J}}{M_k(s)} \langle k \rangle^{2m} E_k \dd k\\
		&= \int_{\mathbb{R}} c \lambda_k \frac{\langle c \lambda_k s \rangle^{2J}}{M_k(s)} \langle k \rangle^{2m} E_k \dd k + \int_{\mathbb{R}} \frac{M_k'(s)}{M_k(s)} \frac{\langle c \lambda_k s \rangle^{2J}}{M_k(s)} \langle k \rangle^{2m} E_k \dd k.
	\end{align*}
	Thus we observe that
	\[\frac{\dd \mathcal{E}_1}{\dd t} \le -4c \widetilde{\mathcal{D}} + \mathcal{NL}_1.\]
	Inserting the above equation and \eqref{eq 2.8}, \eqref{eq 2.9} from Lemma \ref{lem 2.3} into \eqref{eq 3.1}, we get
	\[\mathcal{E}(t) + \mathcal{D}_2(t) \le 2\mathcal{E}(0) - 4c \mathcal{D}_1(t) + \int_{0}^{t} \mathcal{NL}_1 \dd s + 2 \sup_{k \in \mathbb{R}} \left( \int_{0}^{t} \re \langle y \omega_k, y NL_k \rangle + 2 \re \langle \nabla_k \psi_k , \nabla_k \Delta_k^{-1} NL_k \rangle \dd s \right),\]
	which implies
	\begin{equation}
		\label{eq 3.2}
		\mathcal{E}(t) \le 2\mathcal{E}(0) - 4c \mathcal{D}(t) + \mathbb{NL}(t),
	\end{equation}
	since \(c\) is sufficiently small, where
	\[\mathbb{NL}(t) = \int_{0}^{t} \mathcal{NL}_1 \dd s + 2 \sup_{k \in \mathbb{R}} \int_{0}^{t} \re \langle y \omega_k, y NL_k \rangle \dd s + 4 \sup_{k \in \mathbb{R}} \int_{0}^{t} \re \langle \nabla_k \psi_k , \nabla_k \Delta_k^{-1} NL_k \rangle \dd s.\]
	Thus it suffices to give the bound of \(\mathbb{NL}\). In order to estimate \(\mathbb{NL}\), we separate it into the following terms
	\[T_1 = \int_{0}^{t} \int_{\mathbb{R}} \frac{\langle c \lambda_k s \rangle^{2J}}{M_k(s)} \langle k \rangle^{2m} \re \langle \omega_k, NL_k \rangle \dd k \dd s,\]
	\[T_2 = c_\alpha \int_{0}^{t} \int_{\mathbb{R}} \frac{\langle c \lambda_k s \rangle^{2J}}{M_k(s)} \langle k \rangle^{2m} \alpha_k \re \langle \nabla_k \omega_k, \nabla_k NL_k \rangle \dd k \dd s,\]
	\[T_3 = 2 c_\beta \int_{0}^{t} \int_{\mathbb{R}} \frac{\langle c \lambda_k s \rangle^{2J}}{M_k(s)} \langle k \rangle^{2m} \beta_k \re \langle iky NL_k, \partial_y \omega_k \rangle \dd k \dd s,\]
	\[T_4 = 2 c_\beta \int_{0}^{t} \int_{\mathbb{R}} \frac{\langle c \lambda_k s \rangle^{2J}}{M_k(s)} \langle k \rangle^{2m} \beta_k \re \langle iky \omega_k, \partial_y NL_k \rangle \dd k \dd s,\]
	\[T_5 = c_\gamma \int_{0}^{t} \int_{\mathbb{R}} \frac{\langle c \lambda_k s \rangle^{2J}}{M_k(s)} \langle k \rangle^{2m} \gamma_k \re \langle y \omega_k, y NL_k \rangle \dd k \dd s,\]
	\[T_6 = 2 c_\gamma \int_{0}^{t} \int_{\mathbb{R}} \frac{\langle c \lambda_k s \rangle^{2J}}{M_k(s)} \langle k \rangle^{2m} \gamma_k \re \langle \nabla_k \psi_k , \nabla_k \Delta_k^{-1} NL_k \rangle \dd k \dd s,\]
	\[T_7 = 2 \sup_{k \in \mathbb{R}} \int_{0}^{t} \re \langle y \omega_k, y NL_k \rangle \dd s,\]
	\[T_8 = 4 \sup_{k \in \mathbb{R}} \int_{0}^{t} \re \langle \nabla_k \psi_k , \nabla_k \Delta_k^{-1} NL_k \rangle \dd s.\]
	
	\subsection{Technical Lemmas}
	
	Now, we give some minor technical lemmas will be used throughout the eatimates of \(\mathbb{NL}\). The proofs rely on basic inequalities of analysis, namely H\"{o}lder's inequality and Gagliardo-Nirenberg-Sobolev inequality.
	
	\begin{lemma}
		\label{lem 3.4}
		The following estimates hold.
		\begin{equation}
			\label{eq 3.3}
			\int_{\mathbb{R}} \| \nabla_k \omega_k \|_\infty \dd k \lesssim \nu^{-\frac{5}{6}} \widetilde{\mathcal{D}}^\frac{1}{2},
		\end{equation}
		\begin{equation}
			\label{eq 3.4}
			\int_{\mathbb{R}} \| \nabla_k \psi_k \|_\infty \dd k \lesssim \mathcal{E}^\frac{1}{2},
		\end{equation}
	\end{lemma}
	
	\begin{proof}
		To estimate \(\| \cdot \|_\infty\), it is natural to use Gagliardo-Nirenberg-Sobolev inequality,
		\[\int_{\mathbb{R}} \| \nabla_k \omega_k \|_\infty \dd k \lesssim \int_{\mathbb{R}} \| \nabla_k \omega_k \|_2^\frac{1}{2} \| \partial_y \nabla_k \omega_k \|_2^\frac{1}{2} \dd k\\
		\lesssim \int_{\mathbb{R}} \| \nabla_k \omega_k \|_2^\frac{1}{2} \| \Delta_k \omega_k \|_2^\frac{1}{2} \dd k.\]
		By H\"{o}lder's inequality, we obtain
		\[\int_{\mathbb{R}} \| \nabla_k \omega_k \|_\infty \dd k \lesssim \left( \int_{\mathbb{R}} \langle k \rangle^{-2m} \alpha_k^{-\frac{1}{2}} \nu^{-1} \dd k \right)^\frac{1}{2} \left( \int_{\mathbb{R}} \langle k \rangle^{2m} \nu \| \nabla_k \omega_k \|_2^2 \dd k \right)^\frac{1}{4} \left( \int_{\mathbb{R}} \langle k \rangle^{2m} \alpha_k \nu \| \Delta_k \omega_k \|_2^2 \dd k \right)^\frac{1}{4}.\]
		For the first term, we split the integral domain into \(|k| \ge \nu^{-\frac{1}{3}}\) and \(|k| < \nu^{-\frac{1}{3}}\), then the \(|k| \ge \nu^{-\frac{1}{3}}\) term obeys
		\[\int_{|k| \ge \nu^{-\frac{1}{3}}} \langle k \rangle^{-2m} \alpha_k^{-\frac{1}{2}} \nu^{-1} \dd k \approx \int_{|k| \ge \nu^{-\frac{1}{3}}} |k|^{-2m} \left( \nu^\frac{1}{2} |k|^{-\frac{1}{2}} \right)^{-\frac{1}{2}} \nu^{-1} \dd k \approx \nu^{\frac{2}{3}m-\frac{5}{3}} \le \nu^{-\frac{5}{3}},\]
		where we have used \(m > \frac{3}{4}\) and \(\nu < 1\). Then we integrate over \(|k| < \nu^{-\frac{1}{3}}\),
		\[\int_{|k| < \nu^{-\frac{1}{3}}} \langle k \rangle^{-2m} \alpha_k^{-\frac{1}{2}} \nu^{-1} \dd k \lesssim \int_{|k| < \nu^{-\frac{1}{3}}} \left( \nu^\frac{2}{3} \right)^{-\frac{1}{2}} \nu^{-1} \dd k \approx \nu^{-\frac{5}{3}},\]
		Combining the two part and recalling the definition of \(\widetilde{\mathcal{D}}\), then \eqref{eq 3.3} follows. The proof of \eqref{eq 3.4} is similar to \eqref{eq 3.3}. Indeed
		\begin{align*}
			\int_{\mathbb{R}} \| \nabla_k \psi_k \|_\infty \dd k &\lesssim \int_{\mathbb{R}} \| \nabla_k \psi_k \|_2^\frac{1}{2} \| \partial_y \nabla_k \psi_k \|_2^\frac{1}{2} \dd k\\
			&\lesssim \int_{\mathbb{R}} \| \nabla_k \psi_k \|_2^\frac{1}{2} \| \omega_k \|_2^\frac{1}{2} \dd k\\
			&\lesssim \left( \int_{\mathbb{R}} \langle k \rangle^{-2m} \gamma_k^{-\frac{1}{2}} \dd k \right)^\frac{1}{2} \left( \int_{\mathbb{R}} \langle k \rangle^{2m} \gamma_k \| \nabla_k \psi_k \|_2^2 \dd k \right)^\frac{1}{4} \left( \int_{\mathbb{R}} \langle k \rangle^{2m} \| \omega_k \|_2^2 \dd k \right)^\frac{1}{4}.
		\end{align*}
		We still split the integral domain into \(|k| \ge \nu^{-\frac{1}{3}}\) and \(|k| < \nu^{-\frac{1}{3}}\), and recall \(m > \frac{3}{4}\) and \(\nu < 1\), then we have
		\[\int_{|k| \ge \nu^{-\frac{1}{3}}} \langle k \rangle^{-2m} \gamma_k^{-\frac{1}{2}} \dd k \approx \int_{|k| \ge \nu^{-\frac{1}{3}}} |k|^{-2m} \left( \nu^{-\frac{1}{2}} |k|^\frac{1}{2} \right)^{-\frac{1}{2}} \dd k \approx \nu^{\frac{2}{3}m} \le 1,\]
		and
		\[\int_{|k| < \nu^{-\frac{1}{3}}} \langle k \rangle^{-2m} \gamma_k^{-\frac{1}{2}} \dd k \lesssim \int_{|k| < \nu^{-\frac{1}{3}}} \left( \nu^{-\frac{2}{3}} \right)^{-\frac{1}{2}} \dd k \approx 1.\]
		Hence \eqref{eq 3.4} follows.
	\end{proof}
	
	\begin{lemma}
		\label{lem 3.5}
		The following estimate holds.
		\[\left( \int_{0}^{t} \left( \int_{\mathbb{R}} \| \nabla_k \psi_k \|_\infty \dd k \right)^2 \dd s \right)^\frac{1}{2} \lesssim \int_{\mathbb{R}} \left( \int_{0}^{t} \| \nabla_k \psi_k \|_\infty^2 \dd s \right)^\frac{1}{2} \dd k \lesssim \nu^{-\frac{2}{3}} \mathcal{D}^\frac{1}{2}.\]
	\end{lemma}
	
	\begin{proof}
		Using Gagliardo-Nirenberg-Sobolev inequality to estimate \(\| \nabla_k \psi_k \|_\infty\),
		\[\| \nabla_k \psi_k \|_\infty \lesssim |k|^{-\frac{1}{2}} |k|^\frac{1}{2} \| \nabla_k \psi_k \|_2^\frac{1}{2} \| \partial_y \nabla_k \psi_k \|_2^\frac{1}{2} \lesssim |k|^{-\frac{1}{2}} \| \Delta_k \psi_k \|_2 = |k|^{-\frac{1}{2}} \| \omega_k \|_2.\]
		Then applying Minkowski’s inequality for integrals, we have
		\[\left( \int_{0}^{t} \left( \int_{\mathbb{R}} \| \nabla_k \psi_k \|_\infty \dd k \right)^2 \dd s \right)^\frac{1}{2} \lesssim \int_{\mathbb{R}} \left( \int_{0}^{t} \| \nabla_k \psi_k \|_\infty^2 \dd s \right)^\frac{1}{2} \dd k \lesssim \int_{\mathbb{R}} \left( \int_{0}^{t} |k|^{-1} \| \omega_k \|_2^2 \dd s \right)^\frac{1}{2} \dd k.\]
		We split the integral domain into \(|k| \ge \nu^{-\frac{1}{3}}\) and \(|k| < \nu^{-\frac{1}{3}}\), then the \(|k| \ge \nu^{-\frac{1}{3}}\) term obeys
		\begin{align*}
			\int_{|k| \ge \nu^{-\frac{1}{3}}} \left( \int_{0}^{t} |k|^{-1} \| \omega_k \|_2^2 \dd s \right)^\frac{1}{2} \dd k &= \int_{|k| \ge \nu^{-\frac{1}{3}}} \langle k \rangle^{-m} (\gamma_k \nu)^{-\frac{1}{2}} |k|^{-\frac{1}{2}} \left( \int_{0}^{t} \langle k \rangle^{2m} \gamma_k \nu  \| \omega_k \|_2^2 \dd s \right)^\frac{1}{2} \dd k\\
			&\lesssim \left( \int_{|k| \ge \nu^{-\frac{1}{3}}} \langle k \rangle ^{-2m} (\gamma_k \nu)^{-1} |k|^{-1} \dd k \right)^\frac{1}{2} \left( \int_{\mathbb{R}} \int_{0}^{t} \langle k \rangle^{2m} D_k \dd s \dd k \right)^\frac{1}{2}\\
			&\lesssim \nu^{-\frac{1}{6}} \mathcal{D}^\frac{1}{2},
		\end{align*}
		where we have used H\"{o}lder's inequality. Then we integrate over \(|k| < \nu^{-\frac{1}{3}}\).
		\[\int_{|k| < \nu^{-\frac{1}{3}}} \left( \int_{0}^{t} \left( |k|^{-\frac{1}{2}} \| \omega_k \|_2 \right)^2 \dd s \right)^\frac{1}{2} \dd k \lesssim \int_{|k| < \nu^{-\frac{1}{3}}} \nu^{-\frac{1}{2}} |k|^{-\frac{1}{2}} \dd k \cdot \sup_{k \in \mathbb{R}} \left( \int_{0}^{t} \nu \| \omega_k \|_2^2 \dd s \right)^\frac{1}{2} \lesssim \nu^{-\frac{2}{3}} \mathcal{D}^\frac{1}{2}.\]
		Combine the two parts and then the lemma follows.
	\end{proof}
	
	\begin{lemma}
		\label{lem 3.6}
		The following estimate holds.
		\[\int_{\mathbb{R}} \langle c \lambda_k s \rangle^{2J} \langle k \rangle^{2m} |k| \| \nabla_k \psi_k \|_\infty^2 \dd k \lesssim \nu^{-\frac{1}{3}} \widetilde{\mathcal{D}}.\]
	\end{lemma}
	
	\begin{proof}
		By Gagliardo-Nirenberg-Sobolev inequality, we obtain
		\[\int_{\mathbb{R}} \langle c \lambda_k s \rangle^{2J} \langle k \rangle^{2m} |k| \| \nabla_k \psi_k \|_\infty^2 \dd k \lesssim \int_{\mathbb{R}} \langle c \lambda_k s \rangle^{2J} \langle k \rangle^{2m} |k| \| \nabla_k \psi_k \|_2 \| \partial_y \nabla_k \psi_k \|_2 \dd k \lesssim \int_{\mathbb{R}} \langle c \lambda_k s \rangle^{2J} \langle k \rangle^{2m} \| \omega_k \|_2^2 \dd k.\]
		Notice \(\gamma_k \ge \nu^{-\frac{2}{3}}\), which implies the lemma.
	\end{proof}
	
	\subsection{Bound on \(T_1\)}
	
	We introduce the following frequency decomposition
	\[1 = 1_{|k-k'| \le \frac{|k|}{2}} + 1_{|k-k'| > \frac{|k|}{2}},\]
	and correspondingly \(T_1 \eqqcolon T_{1,LH} + T_{1,HL}\). If \(|k-k'| \le \frac{|k|}{2}\), then \(|k| \approx |k'|\). Using the fact and applying Cauchy-Schwarz inequality in \(y\),
	\[|T_{1,LH}| \lesssim \int_{0}^{t} \iint_{\mathbb{R}^2} \langle c \lambda_k s \rangle^J \langle k \rangle^m \| \omega_k \|_2 \langle c \lambda_{k'} s \rangle^J \langle k' \rangle^m \| \nabla_{k'} \omega_{k'} \|_2 \| \nabla_{k-k'}^\bot \psi_{k-k'} \|_\infty \dd k \dd k' \dd s,\]
	where we have used the property that \(\lambda_k\) is monotonically increasing with respect to \(|k|\). Then we use Young's convolution inequality to place the \(k\) and \(k'\) factors into \(L^2\) and the \(k-k'\) factor into \(L^1\),
	\begin{align*}
		|T_{1,LH}| &\lesssim \int_{0}^{t} \left( \int_{\mathbb{R}} \langle c \lambda_k s \rangle^{2J} \langle k \rangle^{2m} \| \omega_k \|_2^2 \dd k \right)^\frac{1}{2} \left( \int_{\mathbb{R}} \langle c \lambda_k s \rangle^{2J} \langle k \rangle^{2m} \| \nabla_k \omega_k \|_2^2 \dd k \right)^\frac{1}{2} \int_{\mathbb{R}} \| \nabla_k \psi_k \|_\infty \dd k \dd s\\
		&\lesssim \int_{0}^{t} \left( \int_\mathbb{R} \langle c \lambda_k s \rangle^{2J} \langle k \rangle^{2m} (\gamma_k \nu)^{-1} D_k \dd k \right)^\frac{1}{2} \left( \int_\mathbb{R} \langle c \lambda_k s \rangle^{2J} \langle k \rangle^{2m} \nu^{-1} D_k \dd k \right)^\frac{1}{2} \int_{\mathbb{R}} \| \nabla_k \psi_k \|_\infty \dd k \dd s.
	\end{align*} 
	According to Lemma \ref{lem 3.4} and recalling \(\gamma_k \ge \nu^{-\frac{2}{3}}\), we have
	\[|T_{1,LH}| \lesssim \int_{0}^{t} (\nu^{-\frac{1}{3}} \widetilde{\mathcal{D}})^\frac{1}{2} (\nu^{-1} \widetilde{\mathcal{D}})^\frac{1}{2} \mathcal{E}^\frac{1}{2} \dd s \lesssim \nu^{-\frac{2}{3}} \mathcal{D}(t) \sup_{s \in [0,t]}\mathcal{E}(s)^\frac{1}{2}.\]
	Turning to \(T_{1,HL}\), we similarly use Cauchy-Schwarz inequality in \(y\), Young's convolution inequality and Lemma \ref{lem 3.4},
	\begin{align*}
		|T_{1,HL}| &\lesssim \int_{0}^{t} \iint_{\mathbb{R}^2} \langle c \lambda_k s \rangle^J \langle k \rangle^m \| \omega_k \|_2 \langle c \lambda_{k-k'} s \rangle^J \langle k-k' \rangle^m \| \nabla_{k-k'}^\bot \psi_{k-k'} \|_2 \| \nabla_{k'} \omega_{k'} \|_\infty \dd k \dd k' \dd s\\
		&\lesssim \int_{0}^{t} \left( \int_{\mathbb{R}} \langle c \lambda_k s \rangle^{2J} \langle k \rangle^{2m} \| \omega_k \|_2^2 \dd k \right)^\frac{1}{2} \left( \int_{\mathbb{R}} \langle c \lambda_k s \rangle^{2J} \langle k \rangle^{2m} \| \nabla_k \psi_k \|_2^2 \dd k \right)^\frac{1}{2} \int_{\mathbb{R}} \| \nabla_k \omega_k \|_\infty \dd k \dd s\\
		&\lesssim \int_{0}^{t} \left( \int_{\mathbb{R}} \langle c \lambda_k s \rangle^{2J} \langle k \rangle^{2m} (\gamma_k \nu)^{-1} D_k \dd k \right)^\frac{1}{2} \left( \int_{\mathbb{R}} \langle c \lambda_k s \rangle^{2J} \langle k \rangle^{2m} \gamma_k^{-1} E_k \dd k \right)^\frac{1}{2} \nu^{-\frac{5}{6}} \widetilde{\mathcal{D}}^\frac{1}{2} \dd s\\
		&\lesssim \nu^{-\frac{2}{3}} \mathcal{D}(t) \sup_{s \in [0,t]} \mathcal{E}(s)^\frac{1}{2}.
	\end{align*}
	
	\subsection{Bound on \(T_2, T_3\) and \(T_4\)}
	
	To estimate \(T_2\), we apply integration by parts,
	\[T_2 = - c_\alpha \int_{0}^{t} \int_{\mathbb{R}} \frac{\langle c \lambda_k s \rangle^{2J}}{M_k(s)} \langle k \rangle^{2m} \alpha_k \re \langle \Delta_k \omega_k, NL_k \rangle \dd k \dd s.\]
	Using the same frequency decomposition as \(T_1\), we note that \(\alpha_k \le \nu^\frac{2}{3}\), and the estimates are similar to \(T_1\). The bound of \(T_2\) is 
	\[|T_2| \lesssim \nu^{-\frac{2}{3}} \mathcal{D}(t) \sup_{s \in [0,t]} \mathcal{E}(s)^\frac{1}{2}.\]
	
	For \(T_3\), we still use the frequency decomposition \(1 = 1_{|k-k'| \le \frac{|k|}{2}} + 1_{|k-k'| > \frac{|k|}{2}}\). Recalling the definition of \(\beta_k\), we obtain \(\beta_k|k| \le 1\), and then, similarly to \(T_1\), we have 
	\[|T_3| \lesssim \nu^{-\frac{2}{3}} \mathcal{D}(t) \sup_{s \in [0,t]} \mathcal{E}(s)^\frac{1}{2}.\]
	
	We turn our attention now to \(T_4\). We use integration by parts and separate \(T_4\) into two parts,
	\begin{align*}
		T_4 &= -2 c_\beta \int_{0}^{t}  \int_{\mathbb{R}} \frac{\langle c \lambda_k s \rangle^{2J}}{M_k(s)} \langle k \rangle^{2m} \beta_k \re \langle \partial_y (iky \omega_k), NL_k \rangle \dd k \dd s\\
		&= -2 c_\beta \int_{0}^{t} \int_{\mathbb{R}} \frac{\langle c \lambda_k s \rangle^{2J}}{M_k(s)} \langle k \rangle^{2m} \beta_k \re \langle iky \partial_y \omega_k, NL_k \rangle \dd k \dd s -2 c_\beta \int_{0}^{t} \int_{\mathbb{R}} \frac{\langle c \lambda_k s \rangle^{2J}}{M_k(s)} \langle k \rangle^{2m} \beta_k \re \langle ik \omega_k, NL_k \rangle \dd k \dd s\\
		&\eqqcolon T_4^1 + T_4^2.
	\end{align*}
	For \(T_4^1\), we note 
	\[|T_4^1| = |T_3| \lesssim \nu^{-\frac{2}{3}} \mathcal{D}(t) \sup_{s \in [0,t]} \mathcal{E}(s)^\frac{1}{2}.\]
	Since \(\beta_k |k| \le 1\), by the estimate of \(T_1\), we have \[|T_4^2| \lesssim \nu^{-\frac{2}{3}} \mathcal{D}(t) \sup_{s \in [0,t]} \mathcal{E}(s)^\frac{1}{2}.\]
	
	\subsection{Bound on \(T_5\)}
	
	\(T_5\) requires a more refined frequency decomposition. Namely, we use the following decomposition
	\[1 = 1_{|k-k'| \le \frac{|k|}{2}} 1_{|k| \ge \nu^{-\frac{1}{3}}} + 1_{|k-k'| \le \frac{|k|}{2}} 1_{|k| < \nu^{-\frac{1}{3}}} + 1_{|k-k'| > \frac{|k|}{2}} 1_{|k-k'| \ge \nu^{-\frac{1}{3}}} + 1_{|k-k'| > \frac{|k|}{2}} 1_{|k-k'| < \nu^{-\frac{1}{3}}},\]
	and correspondingly \(T_5 \eqqcolon T_{5,LH,H} + T_{5,LH,L} + T_{5,HL,H''} + T_{5,HL,L''}\). For \(T_{5,LH,H}\), we use \(|k| \approx |k'|\), Cauchy-Schwarz inequality in \(y\), Young's convolution inequality and Lemma \ref{lem 3.4} as before,
	\begin{align*}
		|T_{5,LH,H}| &\lesssim \int_{0}^{t} \iint_{|k| \ge \nu^{-\frac{1}{3}}} \langle c \lambda_k s \rangle^J \langle k \rangle^m \gamma_k \| y \omega_k \|_2 \langle c \lambda_{k'} s \rangle^J \langle k' \rangle^m \| y \nabla_{k'} \omega_{k'} \|_2 \| \nabla_{k-k'}^\bot \psi_{k-k'} \|_\infty \dd k \dd k' \dd s\\
		&\lesssim \int_{0}^{t} \left( \int_{|k| \ge \nu^{-\frac{1}{3}}} \langle c \lambda_k s \rangle^{2J} \langle k \rangle^{2m} \gamma_k^2 \| y \omega_k \|_2^2 \dd k \right)^\frac{1}{2} \left( \int_{\mathbb{R}} \langle c \lambda_k s \rangle^{2J} \langle k \rangle^{2m} \| y \nabla_k \omega_k \|_2^2 \dd k \right)^\frac{1}{2} \int_{\mathbb{R}} \| \nabla_k \psi_k \|_\infty \dd k \dd s\\
		&\lesssim \int_{0}^{t} \left( \int_{|k| \ge \nu^{-\frac{1}{3}}} \langle c \lambda_k s \rangle^{2J} \langle k \rangle^{2m} \gamma_k^2 (\beta_k |k|^2)^{-1} D_k \dd k \right)^\frac{1}{2} \left( \int_{\mathbb{R}} \langle c \lambda_k s \rangle^{2J} \langle k \rangle^{2m} (\gamma_k \nu)^{-1}  D_k \dd k \right)^\frac{1}{2} \mathcal{E}^\frac{1}{2} \dd s.
	\end{align*}
	Notice that if \(|k| \ge \nu^{-\frac{1}{3}}\), then \(\gamma_k^2 (\beta_k |k|^2)^{-1} = \nu^{-1}\), which implies
	\[|T_{5,LH,H}| \lesssim \int_{0}^{t} (\nu^{-1} \widetilde{\mathcal{D}})^\frac{1}{2} (\nu^{-\frac{1}{3}} \widetilde{\mathcal{D}})^\frac{1}{2} \mathcal{E}^\frac{1}{2} \dd s \lesssim \nu^{-\frac{2}{3}} \mathcal{D}(t) \sup_{s \in [0,t]} \mathcal{E}(s)^\frac{1}{2}.\]
	Now we treat \(T_{5,LH,L}\). Since \(|k| < \nu^{-\frac{1}{3}}\), we have \(\gamma_k = \nu^{-\frac{2}{3}}\). Hence
	\begin{align*}
		|T_{5,LH,L}| &\lesssim \int_{0}^{t} \iint_{\mathbb{R}^2} \langle c \lambda_k s \rangle^J \langle k \rangle^m \nu^{-\frac{1}{3}} \gamma_k^\frac{1}{2} \| y \omega_k \|_2 \langle c \lambda_{k'} s \rangle^J \langle k' \rangle^m \| y \nabla_{k'} \omega_{k'} \|_2 \| \nabla_{k-k'}^\bot \psi_{k-k'} \|_\infty \dd k \dd k' \dd s\\
		&\lesssim \int_{0}^{t} \left( \int_{\mathbb{R}} \langle c \lambda_k s \rangle^{2J} \langle k \rangle^{2m} \nu^{-\frac{2}{3}} \gamma_k \| y \omega_k \|_2^2 \dd k \right)^\frac{1}{2} \left( \int_\mathbb{R} \langle c \lambda_k s \rangle^{2J} \langle k \rangle^{2m} \| y \nabla_k \omega_k \|_2^2 \dd k \right)^\frac{1}{2} \int_{\mathbb{R}} \| \nabla_k \psi_k \|_\infty \dd k \dd s\\
		&\lesssim \nu^{-\frac{1}{2}} \int_{0}^{t} \mathcal{E}^\frac{1}{2} \widetilde{\mathcal{D}}^\frac{1}{2} \int_{\mathbb{R}} \| \nabla_k \psi_k \|_\infty \dd k \dd s.
	\end{align*}
	Then we use H\"{o}lder's inequality in \(s\) and Lemma \ref{lem 3.5},
	\[|T_{5,LH,L}| \lesssim \nu^{-\frac{1}{2}} \sup_{s \in [0,t]} \mathcal{E}(s)^\frac{1}{2} \left( \int_{0}^{t} \widetilde{\mathcal{D}} \dd s \right)^\frac{1}{2} \left( \int_{0}^{t} \left( \int_{\mathbb{R}} \| \nabla_k \psi_k \|_\infty \dd k \right)^2 \dd s \right)^\frac{1}{2} \lesssim \nu^{-\frac{7}{6}} \mathcal{D}(t) \sup_{s \in [0,t]} \mathcal{E}(s)^\frac{1}{2}.\]
	With the LH terms bounded, we now turn to estimate the HL terms, starting with \(T_{5,HL,H''}\). We have \(\gamma_k \lesssim \gamma_{k-k'}\), since \(|k| \lesssim |k-k'|\). Thus
	\begin{align*}
		|T_{5,HL,H''}| &\lesssim \int_{0}^{t} \iint_{|k-k'| \ge \nu^{-\frac{1}{3}}} \langle c \lambda_k s \rangle^J \langle k \rangle^m \gamma_k^\frac{1}{2} \| y \omega_k \|_2 \langle c \lambda_{k-k'} s \rangle^J \langle k-k' \rangle^m \gamma_{k-k'}^\frac{1}{2} \| \nabla_{k-k'}^\bot \psi_{k-k'} \|_\infty \| y \nabla_{k'} \omega_{k'} \|_2  \dd k \dd k' \dd s\\
		&\lesssim \int_{0}^{t} \left( \int_{\mathbb{R}} \langle c \lambda_k s \rangle^{2J} \langle k \rangle^{2m} \gamma_k \| y \omega_k \|_2^2 \dd k \right)^\frac{1}{2} \left( \int_{|k| \ge \nu^{-\frac{1}{3}}} \langle c \lambda_k s \rangle^{2J} \langle k \rangle^{2m} \gamma_k \| \nabla_k \psi_k \|_\infty^2 \dd k \right)^\frac{1}{2} \int_{\mathbb{R}} \| y \nabla_k \omega_k \|_2 \dd k \dd s\\
		&\eqqcolon \int_{0}^{t} I_1 \cdot I_2 \cdot I_3 \dd s.
	\end{align*}
	According to the definition of \(\mathcal{E}\), we have \(I_1 \lesssim \mathcal{E}^\frac{1}{2}\). For \(I_2\), we use Gagliardo-Nirenberg-Sobolev inequality, together with \(\| k \nabla_k \psi_k \|_2 \lesssim \| \Delta_k \psi_k \|_2 = \| \omega_k \|_2, \| \partial_y \nabla_k \psi_k \| \lesssim \| \Delta_k \psi_k \|_2 = \| \omega_k \|_2\),
	\[I_2 \lesssim \left( \int_{|k| \ge \nu^{-\frac{1}{3}}} \langle c \lambda_k s \rangle^{2J} \langle k \rangle^{2m} \gamma_k |k|^{-1} |k| \| \nabla_k \psi_k \|_2 \| \partial_y \nabla_k \psi_k \|_2 \dd k \right)^\frac{1}{2} \lesssim \left( \int_{\mathbb{R}} \langle c \lambda_k s \rangle^{2J} \langle k \rangle^{2m} \gamma_k (\nu^{-\frac{1}{3}})^{-1} \| \omega_k \|_2^2 \dd k \right)^\frac{1}{2} \lesssim \nu^{-\frac{1}{3}} \widetilde{\mathcal{D}}^\frac{1}{2}.\]
	Turning to \(I_3\), by H\"{o}lder's inequality and \(\gamma_k \ge \nu^{-\frac{2}{3}}\), we obtain
	\[I_3 \lesssim \left( \int_{\mathbb{R}} \langle k \rangle^{-2m} (\gamma_k \nu)^{-1} \dd k \right)^\frac{1}{2} \left( \int_{\mathbb{R}} \langle k \rangle^{2m} \gamma_k \nu \| y \nabla_k \omega_k \|_2^2 \dd k \right)^\frac{1}{2} \lesssim \nu^{-\frac{1}{6}} \widetilde{\mathcal{D}}^\frac{1}{2}.\]
	Combining the three parts,
	\[|T_{5,HL,H''}| \lesssim \int_{0}^{t} \mathcal{E}^\frac{1}{2} \nu^{-\frac{1}{3}} \widetilde{\mathcal{D}}^\frac{1}{2} \nu^{-\frac{1}{6}} \widetilde{\mathcal{D}}^\frac{1}{2} \dd s \lesssim \nu^{-\frac{1}{2}} \mathcal{D}(t) \sup_{s \in [0,t]}\mathcal{E}(s)^\frac{1}{2}.\]
	We note that if \(|k-k'| > \frac{|k|}{2}\), then \(|k'| \lesssim |k-k'|\). For \(T_{5,HL,L''}\), \(|k| < 2|k-k'| < 2\nu^{-\frac{1}{3}}\) holds on the domain of integration, which implies \(\gamma_k \approx \nu^{-\frac{2}{3}}\). As a result,
	\begin{align*}
		|T_{5,HL,L''}| \! &\lesssim \! \int_{0}^{t} \! \iint_{|k'| \lesssim \nu^{-\frac{1}{3}}} \! \langle c \lambda_k s \rangle^J \langle k \rangle^m \nu^{-\frac{1}{3}} \gamma_k^\frac{1}{2} \| y \omega_k \|_2 \langle c \lambda_{k-k'} s \rangle^J \langle k-k' \rangle^m |k-k'|^\frac{1}{2} \| \nabla_{k-k'}^\bot \psi_{k-k'} \|_\infty |k'|^{-\frac{1}{2}} \| y \nabla_{k'} \omega_{k'} \|_2 \dd k \dd k' \dd s\\
		&\lesssim \int_{0}^{t} \left( \int_\mathbb{R} \langle c \lambda_k s \rangle^{2J} \langle k \rangle^{2m} \nu^{-\frac{2}{3}} \gamma_k \| y \omega_k \|_2^2 \dd k \right)^\frac{1}{2} \left( \int_{\mathbb{R}} \langle c \lambda_k s \rangle^{2J} \langle k \rangle^{2m} |k| \| \nabla_k \psi_k \|_\infty^2 \dd k \right)^\frac{1}{2} \int_{|k| \lesssim \nu^{-\frac{1}{3}}} |k|^{-\frac{1}{2}} \| y \nabla_k \omega_k \|_2 \dd k \dd s\\
		&\lesssim \sup_{s \in [0,t]} ( \nu^{-\frac{2}{3}} \mathcal{E}(s))^\frac{1}{2} \left( \int_{0}^{t} \nu^{-\frac{1}{3}} \widetilde{\mathcal{D}} \dd s \right)^\frac{1}{2} \left( \int_{0}^{t} \left( \int_{|k| \lesssim \nu^{-\frac{1}{3}}} |k|^{-\frac{1}{2}} \| y \nabla_k \omega_k \|_2 \dd k \right)^2 \dd s \right)^\frac{1}{2},
	\end{align*}
	where we have used Lemma \ref{lem 3.6}. For the last term, using Minkowski’s inequality for integrals,
	\begin{align*}
		\left( \int_{0}^{t} \left( \int_{|k| \lesssim \nu^{-\frac{1}{3}}} |k|^{-\frac{1}{2}} \| y \nabla_k \omega_k \|_2 \dd k \right)^2 \dd s \right)^\frac{1}{2} &\lesssim \int_{|k| \lesssim \nu^{-\frac{1}{3}}} \left( \int_{0}^{t} \left( |k|^{-\frac{1}{2}} \| y \nabla_k \omega_k \|_2 \right)^2 \dd s \right)^\frac{1}{2} \dd k\\
		&\lesssim \int_{|k| \lesssim \nu^{-\frac{1}{3}}} \nu^{-\frac{1}{2}} |k|^{-\frac{1}{2}} \dd k \cdot \sup_{k \in \mathbb{R}} \left( \int_{0}^{t} \nu \| y \nabla_k \omega_k \|_2^2 \dd s \right)^\frac{1}{2} \\
		&\lesssim \nu^{-\frac{2}{3}} \mathcal{D}^\frac{1}{2},
	\end{align*}
	which implies
	\[|T_{5,HL,L''}| \lesssim \nu^{-\frac{7}{6}} \mathcal{D}(t) \sup_{s \in [0,t]} \mathcal{E}(s)^\frac{1}{2}.\]
	
	\subsection{Bound on \(T_6\)}
	
	By integration by parts,
	\[T_6 = -2 c_\gamma \int_{0}^{t} \int_{\mathbb{R}} \frac{\langle c \lambda_k s \rangle^{2J}}{M_k(s)} \langle k \rangle^{2m} \gamma_k \re \langle \psi_k , NL_k \rangle \dd k \dd s.\]
	We use the frequency decomposition
	\begin{align*}
		1 &= 1_{|k-k'| \le \frac{|k|}{2}} 1_{|k| \ge \nu^{-\frac{1}{3}}} + 1_{|k-k'| \le \frac{|k|}{2}} 1_{|k| < \nu^{-\frac{1}{3}}} + 1_{|k-k'| > \frac{|k|}{2}} 1_{|k| \ge \nu^{-\frac{1}{3}}}\\
		&+ 1_{|k-k'| > \frac{|k|}{2}} 1_{|k| < \nu^{-\frac{1}{3}}} 1_{|k| \le |k'|} + 1_{|k-k'| > \frac{|k|}{2}} 1_{|k| < \nu^{-\frac{1}{3}}} 1_{|k| > |k'|},
	\end{align*}
	and correspondingly we write \(T_6 \eqqcolon T_{6,LH,H} + T_{6,LH,L} + T_{6,HL,H} + T_{6,HL,L,LH} + T_{6,HL,L,HL}\). For \(T_{6,LH,H}\), by Cauchy-Schwarz inequality in \(y\) and Young's convolution inequality,
	\begin{align*}
		|T_{6,LH,H}| &\lesssim \int_{0}^{t} \iint_{|k| \ge \nu^{-\frac{1}{3}}} \langle c \lambda_k s \rangle^J \langle k \rangle^m \gamma_k \| \psi_k \|_2 \langle c \lambda_{k'} s \rangle^J \langle k' \rangle^m \| \nabla_{k'} \omega_{k'} \|_2 \| \nabla_{k-k'}^\bot \psi_{k-k'} \|_\infty \dd k \dd k' \dd s\\
		&\lesssim \int_{0}^{t} \left( \int_{|k| \ge \nu^{-\frac{1}{3}}} \langle c \lambda_k s \rangle^{2J} \langle k \rangle^{2m} \gamma_k^2 \| \psi_k \|_2^2 \dd k \right)^\frac{1}{2} \left( \int_{\mathbb{R}} \langle c \lambda_k s \rangle^{2J} \langle k \rangle^{2m} \| \nabla_k \omega_k \|_2^2 \dd k \right)^\frac{1}{2} \int_{\mathbb{R}} \| \nabla_k \psi_k \|_\infty \dd k \dd s.
	\end{align*}
	If \(|k| \ge \nu^{-\frac{1}{3}}\), then \(\beta_k = |k|^{-1}, \gamma_k = \nu^{-\frac{1}{2}}|k|^\frac{1}{2}\). Thus for the first term, we have
	\begin{align*}
		\int_{|k| \ge \nu^{-\frac{1}{3}}} \langle c \lambda_k s \rangle^{2J} \langle k \rangle^{2m} \gamma_k^2 \| \psi_k \|_2^2 \dd k &= \int_{|k| \ge \nu^{-\frac{1}{3}}} \langle c \lambda_k s \rangle^{2J} \langle k \rangle^{2m} \beta_k |k| (\nu^{-\frac{1}{2}}|k|^\frac{1}{2})^2 \| \psi_k \|_2^2 \dd k\\
		&\lesssim \int_\mathbb{R} \langle c \lambda_k s \rangle^{2J} \langle k \rangle^{2m} \nu^{-1} (\nu^{-\frac{1}{3}})^{-2} \beta_k |k|^2 \| k \psi_k \|_2^2 \dd k\\
		&\lesssim \nu^{-\frac{1}{3}} \widetilde{\mathcal{D}},
	\end{align*}
	which implies
	\[|T_{6,LH,H}| \lesssim \int_{0}^{t} (\nu^{-\frac{1}{3}} \widetilde{\mathcal{D}})^\frac{1}{2} (\nu^{-1} \widetilde{\mathcal{D}})^\frac{1}{2} \mathcal{E}^\frac{1}{2} \dd s \lesssim \nu^{-\frac{2}{3}} \mathcal{D}(t) \sup_{s \in [0,t]} \mathcal{E}(s)^\frac{1}{2},\]
	where we have used Lemma \ref{lem 3.4}. We treat \(T_{6,LH,L}\) by separating it into two parts,
	\begin{align*}
		T_{6,LH,L} =& 2 c_\gamma \int_{0}^{t} \int_{|k| < \nu^{-\frac{1}{3}}} \frac{\langle c \lambda_k s \rangle^{2J}}{M_k(s)} \langle k \rangle^{2m} \gamma_k \re \left\langle \psi_k, \int_{|k-k'| \le \frac{|k|}{2}} i(k-k') \psi_{k-k'} \partial_y \omega_{k'} \dd k' \right\rangle \dd k \dd s\\
		&- 2 c_\gamma \int_{0}^{t} \int_{|k| < \nu^{-\frac{1}{3}}} \frac{\langle c \lambda_k s \rangle^{2J}}{M_k(s)} \langle k \rangle^{2m} \gamma_k \re \left\langle \psi_k, \int_{|k-k'| \le \frac{|k|}{2}} \partial_y \psi_{k-k'} ik' \omega_{k'} \dd k' \right\rangle \dd k \dd s\\
		\eqqcolon& T_{6,LH,L}^x + T_{6,LH,L}^y.
	\end{align*}
	We use integration by parts to handle \(T_{6,LH,L}^x\),
	\[T_{6,LH,L}^x = -2 c_\gamma \int_{0}^{t} \iint_{|k-k'| \le \frac{|k|}{2}, |k| < \nu^{-\frac{1}{3}}} \frac{\langle c \lambda_k s \rangle^{2J}}{M_k(s)} \langle k \rangle^{2m} \gamma_k \re \langle \partial_y \psi_k \overline{\psi_{k-k'}} + \psi_k \partial_y \overline{\psi_{k-k'}}, i(k-k') \omega_{k'} \rangle \dd k \dd k' \dd s.\]
	By Cauchy-Schwarz inequality in \(y\),
	\begin{align*}
		|T_{6,LH,L}^x| &\lesssim \int_{0}^{t} \iint_{\mathbb{R}^2} \langle c \lambda_k s \rangle^{2J} \langle k \rangle^{2m} \nu^{-\frac{1}{3}} \gamma_k^\frac{1}{2} (\| \partial_y \psi_k \|_2 |k-k'| \| \psi_{k-k'} \|_\infty + |k| \| \psi_k \|_2 \| \partial_y \psi_{k-k'} \|_\infty) \| \omega_{k'} \|_2 \dd k \dd k' \dd s\\
		&\lesssim \int_{0}^{t} \iint_{\mathbb{R}^2} \langle c \lambda_k s \rangle^J \langle k \rangle^m \nu^{-\frac{1}{3}} \gamma_k^\frac{1}{2} \| \nabla_k \psi_k \|_2 \langle c \lambda_{k'} t \rangle^J \langle k' \rangle^m \| \omega_{k'} \|_2 \| \nabla_{k-k'} \psi_{k-k'} \|_\infty \dd k \dd k' \dd s,		
	\end{align*}
	Then using Young's convolution inequality, H\"{o}lder's inequality in \(s\) and Lemma \ref{lem 3.5}, we obtain
	\begin{align*}
		|T_{6,LH,L}^x| &\lesssim \int_{0}^{t} \left( \int_{\mathbb{R}} \langle c \lambda_k s \rangle^{2J} \langle k \rangle^{2m} \nu^{-\frac{2}{3}} \gamma_k \| \nabla_k \psi_k \|_2^2 \dd k \right)^\frac{1}{2} \left( \int_\mathbb{R} \langle c \lambda_k s \rangle^{2J} \langle k \rangle^{2m} \| \omega_k \|_2^2 \dd k \right)^\frac{1}{2} \int_{\mathbb{R}} \| \nabla_k \psi_k \|_\infty \dd k \dd s\\
		&\lesssim \sup_{s \in [0,t]} (\nu^{-\frac{2}{3}} \mathcal{E}(s))^\frac{1}{2} \left( \int_{0}^{t} (\gamma_k \nu)^{-1} \widetilde{\mathcal{D}} \dd s \right)^\frac{1}{2} \left( \int_{0}^{t} \left( \int_{\mathbb{R}} \| \nabla_k \psi_k \|_\infty \dd k \right)^2 \dd s \right)^\frac{1}{2}\\
		&\lesssim \nu^{-\frac{7}{6}} \mathcal{D}(t) \sup_{s \in [0,t]} \mathcal{E}(s)^\frac{1}{2}.
	\end{align*}
	Turning to \(T_{6,LH,L}^y\), similarly to \(T_{6,LH,L}^x\), we have
	\begin{align*}
		|T_{6,LH,L}^y| &\lesssim \int_{0}^{t} \iint_{\mathbb{R}^2} \langle c \lambda_k s \rangle^J \langle k \rangle^m \nu^{-\frac{1}{3}} \gamma_k^\frac{1}{2} \| \nabla_k \psi_k \|_2 \langle c \lambda_{k'} s \rangle^J \langle k' \rangle^m \| \omega_{k'} \|_2 \| \nabla_{k-k'} \psi_{k-k'} \|_\infty \dd k \dd k' \dd s\\
		&\lesssim \nu^{-\frac{7}{6}} \mathcal{D}(t) \sup_{s \in [0,t]} \mathcal{E}(s)^\frac{1}{2}.
	\end{align*}
	We turn our attention now to the HL terms, starting with \(T_{6,HL,H}\). We use Cauchy-Schwarz inequality in \(y\), Young's convolution inequality as before,
	\begin{align*}
		|T_{6,HL,H}| &\lesssim \int_{0}^{t} \iint_{|k| \ge \nu^{-\frac{1}{3}}} \langle c \lambda_k s \rangle^J \langle k \rangle^m \gamma_k \| \psi_k \|_2 \langle c \lambda_{k-k'} s \rangle^J \langle k-k' \rangle^m \| \nabla_{k-k'}^\bot \psi_{k-k'} \|_2 \| \nabla_{k'} \omega_{k'} \|_\infty \dd k \dd k' \dd s\\
		&\lesssim \int_{0}^{t} \left( \int_{|k| \ge \nu^{-\frac{1}{3}}} \langle c \lambda_k s \rangle^{2J} \langle k \rangle^{2m} \gamma_k^2 \| \psi_k \|_2^2 \dd k \right)^\frac{1}{2} \left( \int_{\mathbb{R}} \langle c \lambda_k s \rangle^{2J} \langle k \rangle^{2m} \| \nabla_k \psi_k \|_2^2 \dd k \right)^\frac{1}{2} \int_{\mathbb{R}} \| \nabla_k \omega_k \|_\infty \dd k \dd s.
	\end{align*}
	As the first term of \(|T_{6,LH,H}|\), we have proved that
	\[\int_{|k| \ge \nu^{-\frac{1}{3}}} \langle c \lambda_k s \rangle^{2J} \langle k \rangle^{2m} \gamma_k^2 \| \psi_k \|_2^2 \dd k \le \nu^{-\frac{1}{3}} \widetilde{\mathcal{D}}.\]
	Combining Lemma \ref{lem 3.4}, we obtain
	\[|T_{6,HL,H}| \lesssim \int_{0}^{t} (\nu^{-\frac{1}{3}} \widetilde{\mathcal{D}})^\frac{1}{2} (\nu^\frac{2}{3} \mathcal{E})^\frac{1}{2} \nu^{-\frac{5}{6}} \widetilde{\mathcal{D}}^\frac{1}{2} \dd s \lesssim \nu^{-\frac{2}{3}} \mathcal{D}(t) \sup_{s \in [0,t]} \mathcal{E}(s)^\frac{1}{2}.\]
	Turning to \(T_{6,HL,L,LH}\), we separate it into two parts, \(T_{6,HL,L,LH}^x\) and \(T_{6,HL,L,LH}^y\), in the same way as we handle \(T_{6,LH,L}\). For \(T_{6,HL,L,LH}^x\), we use integration by parts and separate it into two parts,
	\begin{align*}
		T_{6,HL,L,LH}^x =& -2 c_\gamma \int_{0}^{t} \iint_{|k-k'| > \frac{|k|}{2}, |k| < \nu^{-\frac{1}{3}}, |k| \le |k'|} \frac{\langle c \lambda_k s \rangle^{2J}}{M_k(s)} \langle k \rangle^{2m} \gamma_k \re \langle \partial_y \psi_k \overline{\psi_{k-k'}}, i(k-k') \omega_{k'} \rangle \dd k \dd k' \dd s\\
		&-2 c_\gamma \int_{0}^{t} \iint_{|k-k'| > \frac{|k|}{2}, |k| < \nu^{-\frac{1}{3}}, |k| \le |k'|} \frac{\langle c \lambda_k s \rangle^{2J}}{M_k(s)} \langle k \rangle^{2m} \gamma_k \re \langle \psi_k \partial_y \overline{\psi_{k-k'}}, i(k-k') \omega_{k'} \rangle \dd k \dd k' \dd s\\
		\eqqcolon& T_{6,HL,L,LH}^{x,1} + T_{6,HL,L,LH}^{x,2}.
	\end{align*}
	We treat \(T_{6,HL,L,LH}^{x,1}\) first. Applying Cauchy-Schwarz inequality in \(y\), Young's convolution inequality, H\"{o}lder's inequality in \(s\) and Lemma \ref{lem 3.5}, we obatain
	\begin{align*}
		|T_{6,HL,L,LH}^{x,1}| &\lesssim \int_{0}^{t} \iint_{\mathbb{R}^2} \langle c \lambda_{k-k'} s \rangle^J \langle k-k' \rangle^m \gamma_{k-k'}^\frac{1}{2} \| (k-k') \psi_{k-k'} \|_2 \langle c \lambda_{k'} s \rangle^J \langle k' \rangle^m \gamma_{k'}^\frac{1}{2} \| \omega_{k'} \|_2 \| \partial_y \psi_k \|_\infty \dd k \dd k' \dd s\\
		&\lesssim \int_{0}^{t} \left( \int_{\mathbb{R}} \langle c \lambda_k s \rangle^{2J} \langle k \rangle^{2m} \gamma_k \| \nabla_k \psi_k \|_2^2 \dd k \right)^\frac{1}{2} \left( \int_{\mathbb{R}} \langle c \lambda_k s \rangle^{2J} \langle k \rangle^{2m} \gamma_k \| \omega_k \|_2^2 \dd k \right)^\frac{1}{2} \int_{\mathbb{R}} \| \nabla_k \psi_k \|_\infty \dd k \dd s\\
		&\lesssim \sup_{s \in [0,t]} \mathcal{E}(s)^\frac{1}{2} \left( \int_{0}^{t} \nu^{-1} \widetilde{\mathcal{D}} \dd s \right)^\frac{1}{2} \left( \int_{0}^{t} \left( \int_{\mathbb{R}} \| \nabla_k \psi_k \|_\infty \dd k \right)^2 \dd s \right)^\frac{1}{2}\\
		&\lesssim \nu^{-\frac{7}{6}} \mathcal{D}(t) \sup_{s \in [0,t]} \mathcal{E}(s)^\frac{1}{2}.
	\end{align*}
	Turning to \(T_{6,HL,L,LH}^{x,2}\), we use Cauchy-Schwarz inequality in \(y\), Young's convolution inequality, together with \(\| k \partial_y \psi_k \|_2 \lesssim \| \Delta_k \psi_k \|_2 = \| \omega_k \|_2\),
	\begin{align*}
		|T_{6,HL,L,LH}^{x,2}| &\lesssim \int_{0}^{t} \iint_{|k| < \nu^{-\frac{1}{3}}} \langle c \lambda_{k-k'} s \rangle^J \langle k-k' \rangle^m \gamma_{k-k'}^\frac{1}{2} \| (k-k') \partial_y \psi_{k-k'} \|_2 \langle c \lambda_{k'} s \rangle^J \langle k' \rangle^m \gamma_{k'}^\frac{1}{2} \| \omega_{k'} \|_2 \| \psi_k \|_\infty \dd k \dd k' \dd s\\
		&\lesssim \int_{0}^{t} \left( \int_{\mathbb{R}} \langle c \lambda_k s \rangle^{2J} \langle k \rangle^{2m} \gamma_k \| \omega_k \|_2^2 \dd k \right)^\frac{1}{2} \left( \int_{\mathbb{R}} \langle c \lambda_k s \rangle^{2J} \langle k \rangle^{2m} \gamma_k \| \omega_k \|_2^2 \dd k \right)^\frac{1}{2} \int_{|k| < \nu^{-\frac{1}{3}}} \| \psi_k \|_\infty \dd k \dd s.
	\end{align*}
	To estimate the last term, we use Gagliardo-Nirenberg-Sobolev inequality,
	\[\int_{|k| < \nu^{-\frac{1}{3}}} \| \psi_k \|_\infty \dd k \lesssim \int_{|k| < \nu^{-\frac{1}{3}}} |k|^{-\frac{1}{2}} |k|^\frac{1}{2} \| \psi_k \|_2^\frac{1}{2} \| \partial_y \psi_k \|_2^\frac{1}{2} \dd k \lesssim \int_{|k| < \nu^{-\frac{1}{3}}} |k|^{-\frac{1}{2}} \dd k \cdot \sup_{k \in \mathbb{R}} \| \nabla_k \psi_k \|_2 \lesssim \nu^{-\frac{1}{6}} \mathcal{E}^\frac{1}{2}.\]
	Thus we obtain
	\[|T_{6,HL,L,LH}^{x,2}| \lesssim \int_{0}^{t} \nu^{-1} \widetilde{\mathcal{D}} \nu^{-\frac{1}{6}} \mathcal{E}^\frac{1}{2} \dd s \lesssim \nu^{-\frac{7}{6}} \mathcal{D}(t) \sup_{s \in [0,t]} \mathcal{E}(s)^\frac{1}{2}.\]
	For \(T_{6,HL,L,LH}^y\), similarly to \(T_{6,HL,L,LH}^{x,2}\), we have
	\begin{align*}
		|T_{6,HL,L,LH}^y| &\lesssim \int_{0}^{t} \iint_{|k| < \nu^{-\frac{1}{3}}} \langle c \lambda_{k-k'} s \rangle^J \langle k-k' \rangle^m \gamma_{k-k'}^\frac{1}{2} \| (k-k') \partial_y \psi_{k-k'} \|_2 \langle c \lambda_{k'} s \rangle^J \langle k' \rangle^m \gamma_{k'}^\frac{1}{2} \| \omega_{k'} \|_2 \nu^{-\frac{2}{3}} \| \psi_k \|_\infty \dd k \dd k' \dd s\\
		&\lesssim \nu^{-\frac{7}{6}} \mathcal{D}(t) \sup_{s \in [0,t]} \mathcal{E}(s)^\frac{1}{2},
	\end{align*}
	where we have used \(|k'| \lesssim |k-k'|\) since \(|k-k'| > \frac{|k|}{2}\). Now we treat the last term \(T_{6,HL,L,HL}\). As before, we separate it into two parts, \(T_{6,HL,L,HL} \eqqcolon T_{6,HL,L,HL}^x + T_{6,HL,L,HL}^y\). We still separate \(T_{6,HL,L,HL}^x\) into two parts by integration by parts,
	\begin{align*}
		T_{6,HL,L,HL}^x =& -2 c_\gamma \int_{0}^{t} \iint_{|k-k'| > \frac{|k|}{2}, |k| < \nu^{-\frac{1}{3}}, |k| > |k'|} \frac{\langle c \lambda_k s \rangle^{2J}}{M_k(s)} \langle k \rangle^{2m} \gamma_k \re \langle \partial_y \psi_k \overline{\psi_{k-k'}}, i(k-k') \omega_{k'} \rangle \dd k \dd k' \dd s\\
		&-2 c_\gamma \int_{0}^{t} \iint_{|k-k'| > \frac{|k|}{2}, |k| < \nu^{-\frac{1}{3}}, |k| > |k'|} \frac{\langle c \lambda_k s \rangle^{2J}}{M_k(t)} \langle k \rangle^{2m} \gamma_k \re \langle \psi_k \partial_y \overline{\psi_{k-k'}}, i(k-k') \omega_{k'} \rangle \dd k \dd k' \dd s\\
		\eqqcolon& T_{6,HL,L,HL}^{x,1} + T_{6,HL,L,HL}^{x,2}.
	\end{align*}
	To estimate \(T_{6,HL,L,HL}^{x,1}\), notice \(|k'| \lesssim |k-k'|\) on the domain of integration, and we apply Cauchy-Schwarz inequality in \(y\), Young's convolution inequality, H\"{o}lder's inequality in \(s\) and Lemma \ref{lem 3.6},
	\begin{align*}
		|T_{6,HL,L,HL}^{x,1}| \! &\lesssim \! \int_{0}^{t} \! \iint_{|k'| \lesssim \nu^{-\frac{1}{3}}} \! \langle c \lambda_k s \rangle^J \langle k \rangle^m \nu^{-\frac{1}{3}} \gamma_k^\frac{1}{2} \| \partial_y \psi_k \|_2 \langle c \lambda_{k-k'} s \rangle^J \langle k-k' \rangle^m |k-k'|^\frac{1}{2} \| (k-k') \psi_{k-k'} \|_\infty |k'|^{-\frac{1}{2}} \| \omega_{k'} \|_2 \dd k \dd k' \dd s\\
		&\lesssim \int_{0}^{t} \left( \int_{\mathbb{R}} \langle c \lambda_k s \rangle^{2J} \langle k \rangle^{2m} \nu^{-\frac{2}{3}} \gamma_k \| \nabla_k \psi_k \|_2 \dd k \right)^\frac{1}{2} \left( \int_{\mathbb{R}} \langle c \lambda_k s \rangle^{2J} \langle k \rangle^{2m} |k| \| \nabla_k \psi_k \|_\infty^2 \dd k \right)^\frac{1}{2} \int_{|k| \lesssim \nu^{-\frac{1}{3}}} |k|^{-\frac{1}{2}} \| \omega_k \|_2 \dd k \dd s\\
		&\lesssim \sup_{s \in [0,t]} (\nu^{-\frac{2}{3}} \mathcal{E}(s))^\frac{1}{2} \left( \int_{0}^{t} \nu^{-\frac{1}{3}} \widetilde{\mathcal{D}} \dd s \right)^\frac{1}{2} \left( \int_{0}^{t} \left( \int_{|k| \lesssim \nu^{-\frac{1}{3}}} |k|^{-\frac{1}{2}} \| \omega_k \|_2 \dd k \right)^2 \dd s \right)^\frac{1}{2}.
	\end{align*}
	For the last term, by Minkowski’s inequality for integrals,
	\begin{align*}
		\left( \int_{0}^{t} \left( \int_{|k| \lesssim \nu^{-\frac{1}{3}}} |k|^{-\frac{1}{2}} \| \omega_k \|_2 \dd k \right)^2 \dd s \right)^\frac{1}{2} &\lesssim \int_{|k| \lesssim \nu^{-\frac{1}{3}}} \left( \int_{0}^{t} \left( |k|^{-\frac{1}{2}} \| \omega_k \|_2 \right)^2 \dd s \right)^\frac{1}{2} \dd k\\
		&\lesssim \int_{|k| \lesssim \nu^{-\frac{1}{3}}} \nu^{-\frac{1}{2}} |k|^{-\frac{1}{2}} \dd k \cdot \sup_{k \in \mathbb{R}} \left( \int_{0}^{t} \nu \| \omega_k \|_2^2 \right)^\frac{1}{2}\\
		&\lesssim \nu^{-\frac{2}{3}} \mathcal{D}^\frac{1}{2}.
	\end{align*}
	Thus
	\[|T_{6,HL,L,HL}^{x,1}| \lesssim \nu^{-\frac{7}{6}} \mathcal{D}(t) \sup_{s \in [0,t]} \mathcal{E}(s)^\frac{1}{2}.\]
	To handle \(T_{6,HL,L,HL}^{x,2}\), still using Cauchy-Schwarz inequality in \(y\), Young's convolution inequality, H\"{o}lder's inequality in \(s\), together with \(\| k \partial_y \psi_k \|_2 \lesssim \| \Delta_k \psi_k \|_2 = \| \omega_k \|_2\),
	\begin{align*}
		|T_{6,HL,L,HL}^{x,2}| &\lesssim \int_{0}^{t} \iint_{|k'| \lesssim \nu^{-\frac{1}{3}}} \langle c \lambda_k s \rangle^J \langle k \rangle^m \nu^{-\frac{1}{3}} \gamma_k^\frac{1}{2} |k|^\frac{1}{2} \| \psi_k \|_\infty \langle c \lambda_{k-k'} s \rangle^J \langle k-k' \rangle^m \| (k-k') \partial_y \psi_{k-k'} \|_2 |k'|^{-\frac{1}{2}} \| \omega_{k'} \|_2 \dd k \dd k' \dd s\\
		&\lesssim \int_{0}^{t} \left( \int_{\mathbb{R}} \langle c \lambda_k s \rangle^{2J} \langle k \rangle^{2m} \nu^{-\frac{2}{3}} \gamma_k |k| \| \psi_k \|_\infty^2 \dd k \right)^\frac{1}{2} \left( \int_{\mathbb{R}} \langle c \lambda_k s \rangle^{2J} \langle k \rangle^{2m} \| \omega_k \|_2^2 \dd k \right)^\frac{1}{2} \int_{|k| \lesssim \nu^{-\frac{1}{3}}} |k|^{-\frac{1}{2}} \| \omega_k \|_2 \dd k \dd s\\
		&\lesssim \sup_{s \in [0,t]} \left( \int_{\mathbb{R}} \langle c \lambda_k s \rangle^{2J} \langle k \rangle^{2m} \nu^{-\frac{2}{3}} \gamma_k |k| \| \psi_k \|_\infty^2 \dd k \right)^\frac{1}{2} \left( \int_{0}^{t} \nu^{-\frac{1}{3}} \widetilde{\mathcal{D}} \dd s \right)^\frac{1}{2} \left( \int_{0}^{t} \left( \int_{|k| \lesssim \nu^{-\frac{1}{3}}} |k|^{-\frac{1}{2}} \| \omega_k \|_2 \dd k \right)^2 \dd s \right)^\frac{1}{2}.
	\end{align*}
	By Gagliardo-Nirenberg-Sobolev inequality, the first term obeys
	\[\int_{\mathbb{R}} \langle c \lambda_k s \rangle^{2J} \langle k \rangle^{2m} \nu^{-\frac{2}{3}} \gamma_k |k| \| \psi_k \|_\infty^2 \dd k \lesssim \int_{\mathbb{R}} \langle c \lambda_k s \rangle^{2J} \langle k \rangle^{2m} \nu^{-\frac{2}{3}} \gamma_k |k| \| \psi_k \|_2 \| \partial_y \psi_k \|_2 \dd k \lesssim \nu^{-\frac{2}{3}} \mathcal{E}.\]
	For the third term, it is the same as the third term of \(|T_{6,HL,L,HL}^{x,1}|\), so we obtain
	\[|T_{6,HL,L,HL}^{x,2}| \lesssim \nu^{-\frac{7}{6}} \mathcal{D}(t)\sup_{s \in [0,t]} \mathcal{E}(s)^\frac{1}{2}.\]
	The estimate of \(T_{6,HL,L,HL}^y\) is similar to \(T_{6,HL,L,HL}^{x,2}\), and then we have
	\[|T_{6,HL,L,HL}^y| \lesssim \nu^{-\frac{7}{6}} \mathcal{D}(t)\sup_{s \in [0,t]} \mathcal{E}(s)^\frac{1}{2}.\]
	
	\subsection{Bound on \(T_7\)}
	
	Bounding the term is the simplest. We use Cauchy-Schwarz inequality in \(y\),
	\[|T_7| \lesssim \sup_{k \in \mathbb{R}} \int_{0}^{t} \int_{\mathbb{R}} \| y \omega_k \|_2 \| \nabla_{k-k'}^\bot \psi_{k-k'} \|_\infty \| y \nabla_{k'} \omega_{k'} \|_2 \dd k' \dd s.\]
	By Fubini's theorem, we are able to integrate with respect to \(s\) first, and then apply H\"{o}lder's inequality for the integration over \(s\),
	\[|T_7| \lesssim \sup_{k \in \mathbb{R}} \int_{\mathbb{R}} \sup_{s \in [0,t]} \| y \omega_k \|_2 \left( \int_{0}^{t} \| \nabla_{k-k'} \psi_{k-k'} \|_\infty^2 \dd s \right)^\frac{1}{2} \left( \int_{0}^{t} \| y \nabla_{k'} \omega_{k'} \|_2^2 \dd s \right)^\frac{1}{2} \dd k'.\]
	We use Young's convolution inequality to place the \(k\) and \(k'\) factors into \(L^\infty\) and the \(k-k'\) factor into \(L^1\),
	\[|T_7| \lesssim \sup_{k \in \mathbb{R}} \sup_{s \in [0,t]} \| y \omega_k \|_2 \cdot \int_{\mathbb{R}} \left( \int_{0}^{t} \| \nabla_k \psi_k \|_\infty^2 \dd s \right)^\frac{1}{2} \dd k \cdot \sup_{k \in \mathbb{R}} \left( \int_{0}^{t} \| y \nabla_k \omega_k \|_2^2 \dd s \right)^\frac{1}{2}.\]
	Combining Lemma \ref{lem 3.5}, we have
	\[|T_7| \lesssim \nu^{-\frac{7}{6}} \mathcal{D}(t) \sup_{s \in [0,t]} \mathcal{E}(s)^\frac{1}{2}.\]
	
	\subsection{Bound on \(T_8\)}
	
	We handle \(T_8\) using integration by parts,
	\begin{align*}
		T_8 &= 4 \sup_{k \in \mathbb{R}} \int_{0}^{t} - \re \langle \psi_k , NL_k \rangle \dd s\\
		&= 4 \sup_{k \in \mathbb{R}} \int_{0}^{t} \re \left\langle \psi_k, \int_{\mathbb{R}} i(k-k') \psi_{k-k'} \partial_y \omega_{k'} \dd k' \right\rangle - \re \left\langle \psi_k, \int_{\mathbb{R}} \partial_y \psi_{k-k'} ik' \omega_{k'} \dd k' \right\rangle \dd s\\
		&= 4 \sup_{k \in \mathbb{R}} \int_{0}^{t} \int_\mathbb{R} - \re \langle \partial_y \psi_k \overline{\psi_{k-k'}} + \psi_k \partial_y \overline{\psi_{k-k'}},  i(k-k') \omega_{k'} \rangle - \re \left\langle \psi_k, \partial_y \psi_{k-k'} ik' \omega_{k'} \right\rangle \dd k' \dd s\\
		&= 4 \sup_{k \in \mathbb{R}} \int_{0}^{t} \int_\mathbb{R} - \re \langle \partial_y \psi_k,  i(k-k') \psi_{k-k'} \omega_{k'} \rangle - \re \left\langle k \psi_k, i \partial_y \psi_{k-k'} \omega_{k'} \right\rangle \dd k' \dd s.
	\end{align*}
	By Cauchy-Schwarz inequality in \(y\), H\"{o}lder's inequality in \(s\), Young's convolution inequality and Lemma \ref{lem 3.5}, we have
	\begin{align*}
		|T_8| &\lesssim \sup_{k \in \mathbb{R}} \int_{\mathbb{R}} \int_{0}^{t} \| \nabla_k \psi_k \|_2 \| \nabla_{k-k'} \psi_{k-k'} \|_\infty \| \omega_{k'} \|_2 \dd s \dd k'\\
		&\lesssim \sup_{k \in \mathbb{R}} \int_{\mathbb{R}} \sup_{s \in [0,t]} \| \nabla_k \psi_k \|_2 \left( \int_{0}^{t} \| \nabla_{k-k'} \psi_{k-k'} \|_\infty^2 \dd s \right)^\frac{1}{2} \left( \int_{0}^{t} \| \omega_{k'} \|_2^2 \dd s \right)^\frac{1}{2} \dd k'\\
		&\lesssim \sup_{k \in \mathbb{R}} \sup_{s \in [0,t]} \| \nabla_k \psi_k \|_2 \cdot \int_{\mathbb{R}} \left( \int_{0}^{t} \| \nabla_k \psi_k \|_\infty^2 \dd s \right)^\frac{1}{2} \dd k \cdot \sup_{k \in \mathbb{R}} \left( \int_{0}^{t} \| \omega_k \|_2^2 \dd s \right)^\frac{1}{2}\\
		&\lesssim \nu^{-\frac{7}{6}} \mathcal{D}(t) \sup_{s \in [0,t]} \mathcal{E}(s)^\frac{1}{2}.
	\end{align*} 
	
	\subsection{Bound on \(\mathbb{NL}\)}
	
	Combining the estimates above, we obtain
	\[|\mathbb{NL}| \lesssim \nu^{-\frac{7}{6}} \mathcal{D}(t) \sup_{s \in [0,t]} \mathcal{E}(s)^\frac{1}{2}.\]
	Recalling \eqref{eq 3.2},we have the following theorem.
	
	\begin{theorem}
		\label{thm 3.7}
		There exists \(C > 0\) such that
		\begin{equation}
			\label{eq 3.5}
			\mathcal{E}(t) \le 2 \mathcal{E}(0) - 4c \mathcal{D}(t) + C \nu^{-\frac{7}{6}} \mathcal{D}(t) \sup_{s \in [0,t]} \mathcal{E}(s)^\frac{1}{2}.
		\end{equation}
		Thus if \(\mathcal{E}(0) \le c^2 C^{-2} \nu^\frac{7}{3}\), then
		\begin{equation}
			\label{eq 3.6}
			\sup_{t \in [0, +\infty)} \mathcal{E}(t) \le 2 \mathcal{E}(0).
		\end{equation}
	\end{theorem}
	
	\begin{proof}
		\eqref{eq 3.5} follows directly from \eqref{eq 3.2} and the estimate of \(\mathbb{NL}\). Now we suppose that \(\mathcal{E}(0) \le c^2 C^{-2} \nu^\frac{7}{3}\) and \(\mathcal{E}(0) \ne 0\), then we demonstrate \eqref{eq 3.6} by contradiction. Assume that \eqref{eq 3.6} does not hold, that is, there exists \(t_0 \in [0, +\infty)\) such that \(\mathcal{E}(t_0) > 2 \mathcal{E}(0)\). Let
		\[t^* = \inf \{ t \ge 0 | \mathcal{E}(t) = \min \{ 4 \mathcal{E}(0), \mathcal{E}(t_0) \} \} \le t_0 < \infty,\]
		then \(\mathcal{E}(t) \le \min \{ 4 \mathcal{E}(0), \mathcal{E}(t_0) \} \le 4 \mathcal{E}(0)\) for any \(t \in [0,t^*]\) and \(\mathcal{E}(t^*) = \min \{ 4 \mathcal{E}(0), \mathcal{E}(t_0) \} > 2 \mathcal{E}(0)\). However, according to \eqref{eq 3.5},
		\[\mathcal{E}(t^*) \le 2 \mathcal{E}(0) + \left( -4c + C \nu^{-\frac{7}{6}} (4 \mathcal{E}(0))^\frac{1}{2} \right) \mathcal{D}(t) \le 2 \mathcal{E}(0),\]
		which leads to a contradiction.
	\end{proof}
	
	Theorem \ref{thm 1.1} follows directly from Theorem \ref{thm 3.7}.
	
	\section*{Acknowledgements}
	This work was supported by the National Natural Science Foundation of China 12271497. The author would like to thank professor Lifeng Zhao for his guidance.
	
	\bibliographystyle{IEEEtran}
	\bibliography{Stability_of_Poiseuille_Flow}
	
\end{document}